\documentclass[letterpaper]{article}
\usepackage[utf8]{inputenc}
\usepackage[english]{babel}
\usepackage{amsthm,amssymb,latexsym,amsmath, mathrsfs}
\usepackage[all]{xy}
\usepackage{csquotes}

\usepackage[letterpaper, margin=1in, columnsep=30pt]{geometry}

\usepackage{quiver}
\usetikzlibrary{matrix}

\usepackage{color}
\usepackage{tikz-cd}
\usepackage{tikz}
\usepackage{comment}
\usepackage{hyperref}
\usepackage[
    backend=biber,
    style=authoryear,
    citestyle=authoryear,
    hyperref=true,       
    maxcitenames=2,      
    maxbibnames=99       
]{biblatex}
\usepackage{authblk}

\newcommand{\ch}{\operatorname{ch}}

\newcommand{\supp}{\operatorname{supp}}

\newcommand{\cali}{{\mathcal I}}

\newcommand{\calo}{{\mathcal O}}

\newcommand{\calq}{{\mathcal Q}}

\newcommand{\calt}{{\mathcal T}}

\newcommand{\inext}{{\mathcal E}{\it xt}}

\DeclareMathOperator{\Fitt}{Fitt}

\DeclareMathOperator{\coker}{coker}

\DeclareMathOperator{\codim}{{codim}}
\DeclareMathOperator{\rank}{{rk}}
\DeclareMathOperator{\length}{\operatorname{length}}

\DeclareMathOperator{\ext}{{Ext}}

\DeclareMathOperator{\Bour}{Bour}
\DeclareMathOperator{\indeg}{indeg}
\DeclareMathOperator{\gpdim}{gpdim}

\DeclareMathOperator{\Sing}{Sing}
\DeclareMathOperator{\Syz}{Syz}
\DeclareMathOperator{\Mor}{Hom}
\DeclareMathOperator{\Proj}{Proj}

\DeclareMathOperator{\Ann}{\operatorname{Ann}}
\newcommand{\fp}{\mathfrak{p}}
\DeclareMathOperator{\Ass}{\operatorname{Ass}}

\newcommand\restr[2]{{
  \left.\kern-\nulldelimiterspace 
  #1 
  \vphantom{\big|} 
  \right|_{#2} 
  }}

\newtheorem{theorem}{Theorem}
\newtheorem{prop}[theorem]{Proposition}

\newtheorem{Lemma}[theorem]{Lemma}
\newtheorem{cor}[theorem]{Corollary}

\theoremstyle{definition}
\newtheorem{remark}[theorem]{Remark}
\newtheorem{example}[theorem]{Example}
\newtheorem{defi}[theorem]{Definition}

\newtheorem{maintheorem}{Theorem}

\usepackage[capitalize,nameinlink]{cleveref}

\crefname{prop}{proposition}{propositions}
\Crefname{prop}{Proposition}{Propositions}

\crefname{Lemma}{lemma}{lemmas}
\Crefname{Lemma}{Lemma}{Lemmas}

\crefname{cor}{corollary}{corollaries}
\Crefname{cor}{Corollary}{Corollaries}

\crefname{example}{example}{examples}
\Crefname{example}{Example}{Examples}

\crefname{remark}{remark}{remarks}
\Crefname{remark}{Remark}{Remarks}

\crefname{defi}{definition}{definitions}
\Crefname{defi}{Definition}{Definitions}

\crefname{maintheorem}{theorem}{theorems}
\Crefname{maintheorem}{Theorem}{Theorems}

\tikzcdset{scale cd/.style={every label/.append style={scale=#1},
    cells={nodes={scale=#1}}}}

\title{Bourbaki degree of pairs of projective surfaces}
\author[1,2]{Felipe Monteiro}

\affil[1]{IMECC, University of Campinas (UNICAMP), Campinas, SP, CEP 13083-872, Brazil.}
\affil[2]{Université Bourgogne Europe, CNRS, IMB UMR 5584, F-21000 Dijon, France}

\begin{document}

\maketitle

\begin{abstract}

The present work focuses on studying the logarithmic tangent sheaf associated with sequences of two homogeneous polynomials in four variables. We introduce two positive discrete invariants: the invariant $m$ and the Bourbaki degree of a sequence, inspired by the framework of the Bourbaki degree recently developed for projective plane curves by Jardim-Nejad-Simis. The invariant $m$ plays the role of the Tjurina number of plane projective curves and is bounded by a quadratic relation of the degrees. We establish results concerning the interplay of minimal degree for syzygies of the Jacobian matrix and the introduced discrete invariants. Our approach uses tools from foliation theory, taking advantage of the fact that the logarithmic sheaf is, up to a twist, the tangent sheaf of a codimension-one foliation in $\mathbb{P}^3$. We provide examples and classification results for pencils of cubics and for pairs of a quadric and a cubic. In particular, one of the nearly-free examples induces an unstable, non-split tangent sheaf for a codimension-one foliation of degree $3$, answering, in the negative, a conjecture of Calvo-Andrade, Correa and Jardim from $2018$.
\end{abstract}

\tableofcontents


\section{Introduction}

Let $\kappa$ be an algebraically closed field of characteristic zero and let $R = \kappa[x_0, \ldots, x_n]$ be the polynomial ring in $n+1\geq 3$. For an algebraically independent sequence $\sigma = (f_1, \ldots, f_k)$ of homogeneous polynomials in $R$ with degrees $d_f + 1, \ldots, d_k + 1$, respectively, with $d_f \leq \ldots \leq d_k$ one can consider the Jacobian matrix as a map of locally free sheaves on $\mathbb{P}^n \doteq \Proj R$
$$
\nabla \sigma = \begin{pmatrix}
\nabla f_1\\
\nabla f_2\\
\vdots\\
\nabla f_k
\end{pmatrix}: \mathcal{O}_{\mathbb{P}^n}^{\oplus (n+1)} \rightarrow \bigoplus_{i=1}^k \mathcal{O}_{\mathbb{P}^n}(d_i).
$$
The kernel of $\nabla \sigma$, a reflexive sheaf of rank $n+1-k$ on $\mathbb{P}^n$,
is called the \emph{logarithmic tangent sheaf} associated with the sequence $\sigma$. This definition is in analogy (see \cite{faenzi:hal-03271244}) with the case of divisors in $\mathbb{P}^n$, namely for $k = 1$. A sequence $\sigma$ is said to be \emph{free} whenever the sheaf $\calt_\sigma$ splits as a direct sum of line bundles on $\mathbb{P}^n$. In the case of divisors, one has a short exact sequence of the form:
$$
0 \rightarrow \calt_f \rightarrow \mathcal{O}_{\mathbb{P}^n}^{\oplus (n+1)} \rightarrow \mathcal{I}_{J_f}(d) \rightarrow 0,
$$
where $J_f = ( \partial_0 f, \ldots, \partial_n f ) \subseteq R$ is the Jacobian ideal of the homogeneous polynomial $f$. The sheaf $\calt_f$ is the sheaf of $\mathcal{O}_{\mathbb{P}^n}$-modules associated with the graded $R$-module $\Syz(J_f)$ of \emph{Jacobian syzygies} of $f$. For a sequence $\sigma$, the sheaf $\calt_\sigma$ is associated with the graded $R$-module of syzygies of the Jacobian matrix $\nabla \sigma$.

We denote by $e = \indeg(\calt_\sigma)$ the minimum degree for a nonzero syzygy of the matrix $\nabla \sigma$, called the \emph{initial degree}. For $k = n-1$, a choice of global section of minimum degree $\nu \in H^0(\calt_\sigma(e))$ yields a short exact sequence
$$
0 \rightarrow \mathcal{O}_{\mathbb{P}^n}(-e) \xrightarrow{\nu} \calt_\sigma \rightarrow \mathcal{I}_{B_\nu}(e-d) \rightarrow 0
$$
where $B = B_\nu \subset \mathbb{P}^n$ is a pure codimension two scheme which is generically locally a complete intersection. The scheme $B_\nu$ depends on the choice of the syzygy $\nu$, but its Hilbert polynomial is independent of such a choice. In particular, we may consider the degree $\deg(B_\nu) = \Bour(\sigma)$, which we call \emph{Bourbaki degree} of the sequence $\sigma$, in analogy with the case studied for $k = 1$ and $n = 2$ in \cite{jardim2023bourbakidegreeplaneprojective}. By construction, $\Bour(\sigma) = 0$ if and only if $\sigma$ is a free sequence. The goal of this work is to study the concept above for $n = 3$ and $k = 2$.

For $k = 1$, the singular scheme $\Sigma_f$ of $V(f)$, defined by the Jacobian ideal, plays an important role. If $s = \dim \Sigma_f$, the Hilbert polynomial is given by
$$
H(\mathcal{O}_{J_f}(d), t)=H(R/J_f, t) = \frac{\deg(\Sigma_f)}{s!}t^s+O(t^{s-1}),
$$
and, by definition, the leading coefficient $\deg(\Sigma_f)$ is the \emph{degree} of the singular scheme. Moreover, assuming the hypersurface $V(f)$ is reduced, we obtain $s \leq n-2$. For $n = 2$, this means singularities are isolated, and $\deg(\Sigma_f)$ coincides with the Tjurina number of the projective plane curve $V(f) \subset \mathbb{P}^2$. The formula of the Bourbaki degree in this case is (\cite[Theorem 2.1]{jardim2023bourbakidegreeplaneprojective}):
$$
\Bour(f) = e(e-d) + d^2 - \deg(\Sigma_f).
$$

We denote by $\mathcal{Q}_\sigma = \coker(\nabla \sigma)$ the cokernel of the Jacobian matrix and call the scheme $\Xi_\sigma \doteq V( \bigwedge^k \nabla \sigma )$ the \emph{Jacobian scheme} of $\sigma$, given by the zero locus of the $(k \times k)$-minors of the matrix. The set-theoretical support of $\mathcal{Q}_\sigma$ is the reduced locus $(\Xi_\sigma)_{\text{red}}$, although the scheme structure of both may differ. When $\dim(\mathcal{Q}_\sigma) \leq n-2$, we call $\sigma$ a \emph{normal sequence}. 

Hereafter, we assume $n=3$ and $k = 2$, that is, sequences $\sigma = (f,g)$ of homogeneous polynomials $f,g \in R \doteq \kappa[x_0, \ldots, x_3]$ with degrees $\deg(f) = d_f+1, \deg(g) = d_g+1$, and we set $d \doteq d_f + d_g$. We also assume sequences $\sigma$ are \emph{normal}, as defined above, so $c_1(\mathcal{Q}_\sigma) = 0$ and $c_1(\calt_{\sigma}) = -d$. Then $\dim(\mathcal{Q}_\sigma) \leq 1$ and the Hilbert polynomial is of the form
$$
H(\mathcal{Q}_\sigma, t) = m(\sigma) t + O(1),
$$
where $m(\sigma) \doteq \deg(\mathcal{Q}_\sigma) \geq 0$ and $m(\sigma) = 0$ if and only if $\Xi_\sigma$ is zero-dimensional. From the associativity formula (\cite[Theorem $14.7$]{matsumura1989commutative}) applied to the graded homogeneous module $Q$ associated with $\mathcal{Q}_\sigma$, we obtain the expression  
$$
m(\sigma) = \sum_{\fp \in \Ann(\mathcal{Q}_\sigma), \dim(R/\fp) = 2} \length_{R_\fp} Q_\fp \cdot \deg(R/\fp).
$$

Using the fact that $\calt_\sigma$ is a rank two reflexive sheaf, we reproduce analogous considerations as in \cite{jardim2023bourbakidegreeplaneprojective}. We show the following formula for the \emph{Bourbaki degree} of $\sigma$:

\textbf{Proposition:} \textit{Let $\sigma = (f,g)$ be a normal sequence of homogeneous polynomials in $R$, with degrees $d_f + 1, d_g + 1$ respectively, and set $d = d_f + d_g$. If $e = \indeg(\calt_\sigma)$, then the Bourbaki degree of $\sigma$ is given by}
$$
\Bour(\sigma) = e(e-d) + d_f^2 + d_g^2 + d_f d_g - m(\sigma).
$$

A first result of the paper is some bounds for the discrete quantities defined above.

\textbf{Theorem A}(\ref{thmA}) \textit{ Let $\sigma = (f,g)$ be a normal sequence of homogeneous polynomials in $R$, with degrees $d_f + 1, d_g + 1$ respectively, and set $d = d_f + d_g$. Then:
\begin{itemize}
    \item[(a)] $\indeg(\calt_\sigma) \leq d$; 
    \item[(b)] $m(\sigma) \leq d_f^2 + d_g^2 + d_f d_g$;
    \item[(c)] The following are equivalent:
    \begin{itemize}
        \item[(1)] $m(\sigma) = d_f^2 + d_g^2 + d_f d_g$;
        \item[(2)] $\indeg(\calt_\sigma) = 0$;
        \item[(3)] $\calt_\sigma \simeq \mathcal{O}_{\mathbb{P}^3} \oplus \mathcal{O}_{\mathbb{P}^3}(-d)$;
        \item[(4)] $\sigma$ is compressible, that is, up to a linear change of coordinates of $\mathbb{P}^3$, both polynomials in $\sigma$ are independent of some variable $x_0, \ldots, x_3$;
    \end{itemize}
\end{itemize}}

We also show inequalities for $m(\sigma)$ and $\Bour(\sigma)$ in \Cref{subsec:bourbaki-degree-first} related to freeness and $\mu$-stability of the sheaf $\calt_\sigma$ (see \Cref{prop:Bour-lower-bound-stab}). In general, the Bourbaki degree is bounded by $\Bour(\sigma) \leq d_f^2 + d_g^2 + d_f d_g$, and the maximum is attained precisely whenever $m(\sigma) = 0$. In this case, the cokernel $\mathcal{Q}_\sigma$ is minimally resolved by the Buchsbaum--Rim complex. In particular, $\calt_\sigma(d) \simeq \coker(\nabla \sigma)^T$.

In \Cref{subsec:bourbaki-degree-proj-plane-curve}, we consider sequences $\sigma = (x_3, g)$, where $g \in \kappa[x_0, x_1, x_2]$ is a reduced homogeneous polynomial. We show that $\Bour(\sigma)$ coincides with the Bourbaki degree of the plane curve defined by $V(g) \subset \mathbb{P}^2 = V(x_3)$. In \Cref{sec:resolutions}, we relate free resolutions of $\calt_\sigma$ and of the ideal sheaf $\mathcal{I}_B$, where $B$ is the Bourbaki scheme obtained by the choice of a minimal degree syzygy. As a consequence, we introduce the notions of \emph{nearly free sequences} and $3$\emph{-syzygy sequences}, with the following chain of implications
$$
\sigma \text{ is nearly free} \Rightarrow \sigma \text{ is a }3\text{-syzygy sequence}  \Rightarrow \gpdim(\calt_\sigma) = 1,
$$
where $\gpdim(\calt_\sigma)$ is the minimal length of a free resolution of the $R$-module $H^0_*(\calt_\sigma) = \bigoplus_{l \in \mathbb{Z}} H^0(\calt_\sigma(l))$. The converses do not hold, see \Cref{ex:pcubics-pog-not-nf} and \Cref{ex:pencilcubics-Bour4-m5-notpog}.

In \Cref{sec:lower-initial-indeg}, we explore the structure of codimension-one foliation associated with $\sigma$, presented in \cite[Section 9]{faenzi:hal-03271244}, to obtain characterizations of low initial degrees $\indeg(\calt_\sigma) \in \{1, 2\}$. For plane projective curves, $e = 1$ implies that $\Bour(f) \in \{0,1\}$ (see, for example, \cite[Corollary 2.11]{jardim2023bourbakidegreeplaneprojective}). The main theorem of the section is:

\textbf{Theorem B.}(\ref{thmB}) \textit{Let $\sigma = (f,g)$ be a normal sequence of homogeneous polynomials with degrees $d_f + 1, d_g + 1$, respectively. Then:
\begin{itemize}
    \item[(a)] If $\indeg(\calt_\sigma) = 1$, then $\Bour(\sigma) \in \{0,  1, 2\}$;
    \item[(b)] If $\indeg(\calt_\sigma) = 2$, then $\Bour(\sigma) \leq 5$.
\end{itemize}}

We note that the bounds above are independent of the degrees $d_f, d_g$. In \Cref{sec:special-families}, we study on two particular families of normal sequences: pencils of cubics $(d_f = d_g = 2)$ and sequences with $d_f = 1, d_g = 2$, defining a degree $6$ curve inside a quadric surface in $\mathbb{P}^3$. To do so, we use classical results on bounds for Chern classes of reflexive sheaves in $\mathbb{P}^3$. We classify all free and nearly-free cases in these classes in terms of their discrete invariants and establish some stability results, which are summarized at the tables below (for precise statements, see Theorem~\ref{thmC} and Theorem~\ref{thmD}). 

Finally, we provide a counter-example for a conjecture posed by Calvo-Andrade, Correa and Jardim, in \cite{CalvoAndrade-Jardim-Correa-Fol}. They proposed that a tangent sheaf of a codimension-one foliation on $\mathbb{P}^3$ is either split or $\mu$-semistable. \Cref{ex:nearly-free-mixed-degrees-e1} gives a sequence $\sigma$ which induces a codimension-one foliation whose tangent sheaf is neither split nor $\mu$-semistable.

Along the text, we describe examples developed computationally with aid of Macaulay2 software (\cite{M2}) for the two families of sequences.

\newpage 

\begin{center}
Classification results for $d_f = d_g = 2$.
\begin{table}[h]
\begin{tabular}{|l|l|l|l|}
\hline
$m(\sigma)$ & $\Bour(\sigma)$ & Behavior               & minimal resolution for logarithmic tangent sheaf                                                                                                   \\ \hline
12          & 0               & free, compressible     & $\mathcal{O}_{\mathbb{P}^3} \oplus \mathcal{O}_{\mathbb{P}^3}(-4)$                                                                                 \\ \hline
9           & 0               & free, unstable         & $\mathcal{O}_{\mathbb{P}^3}(-1) \oplus \mathcal{O}_{\mathbb{P}^3}(-3)$                                                                             \\ \hline
8           & 0               & free, $\mu$-semistable & $\mathcal{O}_{\mathbb{P}^3}(-2)^{\oplus 2}$                                                                                                        \\ \hline
8           & 1               & nearly-free, $\mu$-unstable  & $0 \to \mathcal{O}_{\mathbb{P}^3}(-5) \to \mathcal{O}_{\mathbb{P}^3}(-1) \oplus \mathcal{O}_{\mathbb{P}^3}(-4)^{\oplus 2} \to \calt_\sigma \to 0$  \\ \hline
7           & 1               & nearly-free, strictly $\mu$-sst     & $0 \to \mathcal{O}_{\mathbb{P}^3}(-4) \to \mathcal{O}_{\mathbb{P}^3}(-2) \oplus \mathcal{O}_{\mathbb{P}^3}(-3)^{\oplus 2} \to \calt_\sigma \to 0$  \\ \hline
$\leq 6$    & $\geq 2$        & $\mu$-semistable       &                                                                                                                                                    \\ \hline
$\leq 2$    & $\geq 6$        & $\mu$-stable           &                                                                                                                                                    \\ \hline
0           & 12              & $\mu$-stable           & $0 \to \mathcal{O}_{\mathbb{P}^3}(-6)^{\oplus 2} \xrightarrow{(\nabla \sigma)^T} \mathcal{O}_{\mathbb{P}^3}(-4)^{\oplus 4} \to \calt_\sigma \to 0$ \\ \hline
\end{tabular}
\end{table}
\end{center}
\begin{center}
Classification results for $d_f = 1$, $d_g = 2$.
\begin{table}[h]
\begin{tabular}{|l|l|l|l|}
\hline
$m(\sigma)$ & $\Bour(\sigma)$ & Behavior                  & minimal resolution for logarithmic tangent sheaf                                                                                                                             \\ \hline
7           & 0               & free, compressible        & $\mathcal{O}_{\mathbb{P}^3} \oplus \mathcal{O}_{\mathbb{P}^3}(-3)$                                                                                                           \\ \hline
5           & 0               & free, unstable            & $\mathcal{O}_{\mathbb{P}^3}(-1) \oplus \mathcal{O}_{\mathbb{P}^3}(-2)$                                                                                                       \\ \hline
4           & 1               & nearly-free, unstable     & $0 \to \mathcal{O}_{\mathbb{P}^3}(-3) \to \mathcal{O}_{\mathbb{P}^3}(-2)^{\oplus 3} \to \calt_\sigma \to 0$                                                                  \\ \hline
4           & 1               & nearly-free, $\mu$-stable & $0 \to \mathcal{O}_{\mathbb{P}^3}(-4) \to \mathcal{O}_{\mathbb{P}^3}(-1) \oplus \mathcal{O}_{\mathbb{P}^3}(-3)^{\oplus 2} \to \calt_\sigma \to 0$                            \\ \hline
$< 3$       & $\geq 3$        & $\mu$-stable              &                                                                                                                                                                              \\ \hline
0           & 7               & $\mu$-stable              & $0 \to \mathcal{O}_{\mathbb{P}^3}(-4)\oplus \mathcal{O}_{\mathbb{P}^3}(-5) \xrightarrow{(\nabla \sigma)^T} \mathcal{O}_{\mathbb{P}^3}(-3)^{\oplus 4} \to \calt_\sigma \to 0$ \\ \hline
\end{tabular}
\end{table}
\end{center}

\subsection*{Acknowledgments}
FM is financed by the São Paulo Research Foundation (FAPESP), Brasil. Process Number \#2021/10550-4. FM is currently a PhD. student under the advice of Marcos Jardim and Daniele Faenzi, with partial funding by the Bridges "Brazil-France interplays in Gauge Theory, extremal structures and stability" projects ANR-21-CE40-0017 and ANR-17-EURE-0002. We thank Marcos Jardim, Daniele Faenzi, Alan Muniz, Victor Cordeiro and Abbas Nejad for fruitful discussions and suggestions regarding this work.


\section{The Bourbaki degree of pairs of projective surfaces}\label[section]{sec:Bourbaki-degree}

In this section, we develop the concept of the \emph{Bourbaki degree} of pairs of projective surfaces on $\mathbb{P}^3$, determined by normal sequences $\sigma = (f,g)$ of homogeneous polynomials. We start with definitions and first results (\Cref{subsec:bourbaki-degree-first}), followed by a reduction to the case of a projective plane curve (\Cref{subsec:bourbaki-degree-proj-plane-curve}) and finish with results relating the geometry of the Bourbaki scheme and the associated logarithmic sheaf, using free resolutions (\Cref{sec:resolutions}).

\subsection{Framework and first results}\label[subsection]{subsec:bourbaki-degree-first}

By a sequence $\sigma = (f,g)$, unless otherwise stated, we mean an algebraically independent sequence of two homogeneous polynomials in $R \doteq \kappa[x_0, \ldots, x_3]$ with degrees $\deg(f) = d_f + 1$, $\deg(g) = d_g + 1$. By \emph{curve} we mean a locally Cohen-Macaulay closed subscheme of $\mathbb{P}^3$ of pure dimension one. 
     
Each sequence $\sigma = (f,g)$ induces a morphism of sheaves on $\mathbb{P}^3$ by the Jacobian matrix:
$$
\nabla \sigma : \mathcal{O}_{\mathbb{P}^3}^{\oplus 4} \rightarrow \mathcal{O}_{\mathbb{P}^3}(d_f) \oplus \mathcal{O}_{\mathbb{P}^3}(d_g),
$$
and we denote the kernel by $\calt_\sigma$, the image by $\mathcal{M}_\sigma$ and the cokernel by $\mathcal{Q}_\sigma$. A sequence $\sigma$ will be called \emph{free} if $\calt_\sigma$ splits as a direct sum of line bundles. We assume throughout that the sequences $\sigma$ satisfy the condition $\codim(\mathcal{Q}_\sigma) \geq 2$, and in this case we call $\sigma$ a \emph{normal sequence}. With this assumption, we obtain $\rank(\mathcal{Q}_\sigma) = 0$ and $c_1(\mathcal{Q}_\sigma) = 0$. The following lemma relates the Hilbert polynomial of $\mathcal{Q}_\sigma$ and its Chern characters.

\begin{Lemma}\label[Lemma]{lem:int-theory-Q}
Let $\mathcal{Q}$ be a coherent sheaf on $\mathbb{P}^3$ with $\rank(\mathcal{Q}) = 0$ and $c_1(\mathcal{Q}) = 0$. Then the Hilbert polynomial of $\mathcal{Q}$ is given by
$$
\mathcal{X}(\mathcal{Q}(t)) = \ch_2(\mathcal{Q}) t + \ch_3(\mathcal{Q}) + 2\ch_2(\mathcal{Q}).
$$
\end{Lemma}

\begin{proof}
This is a direct application of the Hirzebruch-Riemann-Roch Theorem.
\end{proof}

We denote by $m(\sigma) \doteq \ch_2(\mathcal{Q}_\sigma)$. Since $m(\sigma)$ is the leading coefficient of a Hilbert polynomial, $m(\sigma) \geq 0$, and it is zero if and only if $\mathcal{Q}_\sigma$ is a zero-dimensional sheaf.

The \emph{Jacobian scheme} of $\nabla \sigma$ is defined as
$$
\Xi_\sigma \doteq V \left(\bigwedge^2 \nabla \sigma \right) = V(\Fitt_0(\nabla \sigma))
$$
the zero-locus of the $2 \times 2$ minors of $\nabla \sigma$ or, in other words, the $0-$th Fitting scheme of $\nabla \sigma$. From the general theory of Fitting ideals, the annihilator ideal $\Ann \mathcal{Q}_\sigma$ of $\mathcal{Q}_\sigma$ contains the Fitting ideal $\Fitt_0(\nabla \sigma)$, and they have the same reduced support, namely
$$
V(\sqrt{\Ann(\mathcal{Q}_\sigma}) = \supp(\mathcal{Q}_\sigma)_{\text{red}} = V(\Fitt_0(\nabla \sigma))_{\text{red}} = V(\sqrt{\Fitt_0(\nabla \sigma)})
$$
as closed subsets of $\mathbb{P}^3$. In particular, the irreducible components of the schematic support $\supp(\mathcal{Q}_\sigma)$ coincide with the irreducible components of $\Xi_\sigma$. The respective schematic structures may be different, as we explore in the next example. Denote by $(\Xi_\sigma)_1$ the one-dimensional part of the scheme $\Xi_\sigma$.

\begin{example}\label[example]{ex:schematic-difference}
Let $\sigma = (2x_1x_3 - x_1^2,3x_2x_3^2 - 3x_0x_1x_3 + x_1^3)$. The Jacobian matrix is of the form
$$
\begin{pmatrix}
0         & -2x_1+2x_3      & 0     & 2x_1\\
-3x_1x_3   & 3x_1^2-3x_0x_3  & 3x_3^2 & -3x_0x_1+6x_2x_3\\
\end{pmatrix}.
$$
The annihilator ideal of $\coker(\nabla \sigma)$ and the $0$-th Fitting ideal are different, namely given by:
$$
\begin{cases}
\Ann (\mathcal{Q}_\sigma) &= (x_3^2, x_1x_3, x_0x_1^2-x_1^3) = (x_1, x_3)^2 \cap (x_3, x_0-x_1)\\
\Fitt_0(\mathcal{Q}_\sigma) &= (x_3^2, x_1x_3^2, x_1^2x_3, x_0x_1^2-x_1^3-2x_1x_2x_3+2x_2x_3^2),
\end{cases}
$$
so both schemes $\supp(\mathcal{Q}_\sigma)$ and $\Xi_\sigma$ are non-reduced, with degrees $4$ and $6$, respectively. The matrix below
$$
\begin{pmatrix}
x_3 & x_0x_1 - x_1^2 - 2x_2x_3\\
0   & -x_1x_3\\
x_1 & -2x_2x_3 \\
0   & -x_1x_3+x_3^2
\end{pmatrix}
$$
gives two linearly independent syzygies for the matrix $\nabla \sigma$, and thus $\calt_\sigma \simeq \mathcal{O}_{\mathbb{P}^3}(-1) \oplus \mathcal{O}_{\mathbb{P}^3}(-2)$ and $\sigma$ is free. From this and the additivity of the Hilbert polynomial on short exact sequences, one may compute $m(\sigma) = 5$.
\end{example}

The following proposition gives a sufficient condition to forbid the anomaly above to occur.

\begin{prop}\label[prop]{prop:coincidence-m-Xi}
Let $\sigma = (f,g)$ be a normal sequence. If $\codim(\Fitt_1(\nabla \sigma)) \geq 3$, then
$$
m(\sigma) = \deg(\Xi_\sigma) = \deg(\supp \mathcal{Q}_\sigma).
$$
\end{prop}

\begin{proof}
Let $Q_\sigma$ denote the graded $R$-module associated to $\mathcal{Q}_\sigma$. All three invariants above may be described using the associativity formula, as follows:
\begin{align*}
\deg(\Xi) &= \sum_{\fp \in \Ass(\Fitt_0(\nabla \sigma)),\; \dim(R/\fp) = \dim(\Fitt_0(\nabla \sigma)_\fp)} \length_{R_\fp}\bigl((R/\Fitt_0(\nabla \sigma))_\fp\bigr) \cdot \deg(R/\fp) \\
\deg(\supp(\mathcal{Q}_\sigma)) &= \sum_{\fp \in \Ass(\Fitt_0(\nabla \sigma)),\; \dim(R/\fp) = \dim((Q_\sigma)_\fp)} \length_{R_\fp}\bigl((R/\Ann Q_\sigma)_\fp\bigr) \cdot \deg(R/\fp) \\
m(\sigma) &= \sum_{\fp \in \Ass(\Fitt_0(\nabla \sigma)),\; \dim(R/\fp) = \dim((Q_\sigma)_\fp)} \length_{R_\fp}\bigl((Q_\sigma)_\fp\bigr) \cdot \deg(R/\fp).
\end{align*}
The sets of primes involved in each sum are the same, so we only need to check that the lengths coincide over minimal prime ideals $\fp$ of height two. Since $\codim(\Fitt_1(\nabla \sigma)) \geq 3$, the rank of the localization $(\nabla \sigma)_\fp$ is one, and we may change coordinates locally to write
$$
A_\fp \doteq (\nabla \sigma)_\fp = \begin{pmatrix}
1 & 0 & 0 & 0\\
0 & a & b & c
\end{pmatrix}.
$$
In these coordinates, we obtain $\Fitt_0(A_\fp) = (a,b,c)$ from the $2 \times 2$-minors, and moreover $(Q_\sigma)_\fp \simeq R_\fp / (a,b,c)$. We claim that $\Ann(Q_\sigma)_\fp = (a,b,c)$. Since $(a,b,c)=\Fitt_0(\nabla \sigma) \subset \Ann(Q_\sigma)$, we show the other inclusion. Let $r \in R_\fp$ such that $r (Q_\sigma)_{\fp} = 0$. Writing $(Q_\sigma)_{\fp} = \coker(A_\fp)$, we have
$$
r \cdot \begin{pmatrix}
v_1 \\
v_2
\end{pmatrix} = \begin{pmatrix}
r v_1\\
r v_2
\end{pmatrix} \in \text{im }A_\fp
$$
for $(v_1, v_2) \in R_\fp^{2}$. From the second equation, one obtains $
r \cdot v_2 = a u_1 + b u_2 + c u_3$, for some $u_i \in A_p$. Taking $v_2 = 1$ yields $r \in (a,b,c)$. This shows that
$$
\length_{R_\fp}(R_\fp / \Fitt_0(\nabla \sigma)_\fp) = \length_{R_\fp}(R_\fp / \Ann(Q_\sigma)_\fp) = \length_{R_\fp}((Q_\sigma)_\fp),
$$
and since $(Q_\sigma)_\fp \simeq R_\fp/(a,b,c)$, the claim follows.
\end{proof}

Given a sequence $\sigma = (f,g)$ and a saturated syzygy of the Jacobian matrix $\nu$ of degree $l \in \mathbb{Z}$, we obtain a short exact sequence:
$$
0 \rightarrow \mathcal{O}_{\mathbb{P}^3}(-l) \xrightarrow{\nu} \calt_\sigma \rightarrow \cali_{B_{\nu}}(p) \rightarrow 0, 
$$
where $B_\nu \subset \mathbb{P}^{3}$ is the curve associated with $\nu$. Using the Hilbert polynomial, one can write formulas (see \cite[Proposition $4.1$]{Hartshorne1980}):
\begin{align*}
\deg(B_\nu) &=c_2(\calt_{\sigma}(l))\\  c_3(\calt_\sigma) &= 2p_a(B) - 2 + \deg(B)(4+d-2l).
\end{align*}
Since $\calt_\sigma \hookrightarrow \calo_{\mathbb{P}^3}^{\oplus 4}$ and the latter is a $\mu$-semistable sheaf, $l \geq 0$. The following proposition describes a formula for the degree $\deg(B_\nu)$ in terms of the discrete invariants $l$, $d_f$, $d_g$ and $m(\sigma)$.

\begin{prop}\label[prop]{prop:Bour-degree-formula}
Let $\sigma = (f,g)$ be a normal sequence of homogeneous polynomials in $R$ with degrees $\deg(f) = d_f + 1$, $\deg(g) = d_g + 1$. For a saturated syzygy $\nu \in H^0(\calt_\sigma(l))$ of degree $l \geq 0$, let $B_\nu \subset \mathbb{P}^3$ be the associated curve. Then, we have the following equation:
$$
\deg(B_\nu) = l^2 - l(d_f + d_g) + m_0 - m(\sigma),
$$
where $m_0 \doteq d_f^2 + d_g^2 + d_f d_g$.
\end{prop}

\begin{proof}
From the hypothesis $c_1(\calq_\sigma) = 0$, we conclude $c_1(\calt_\sigma) = -(d_f + d_g)$ and $\ch_2(\calq_\sigma) = m(\sigma)$. Moreover, from additivity of $\ch_2$ on short exact sequences:
\begin{align*}
\ch_2(\calt_\sigma)
&= - \ch_2(\calo_{\mathbb{P}^3}(d_f) \oplus \calo_{\mathbb{P}^3}(d_g)) + \ch_2(\calq_\sigma)\\
&= - \frac{d_f^2 + d_g^2}{2} + m(\sigma).
\end{align*}
We can relate the Chern character and the Chern classes by the formula
\begin{align*}
m(\sigma) &= \ch_2(\calt_\sigma) + \frac{d_f^2 + d_g^2}{2}\\ 
&= \frac{c_1^2(\calt_\sigma) - 2c_2(\calt_\sigma)}{2} + \frac{d_f^2 + d_g^2}{2}\\
&= \frac{(d_f + d_g)^2}{2} + \frac{d_f^2 + d_g^2}{2} - c_2(\calt_\sigma)\\
&= d_f^2 + d_g^2 + d_f d_g - c_2(\calt_\sigma),
\end{align*}
so that
\begin{align}\label{eq-1}
c_2(\calt_\sigma) = d_f^2 + d_g^2 + d_f d_g - m(\sigma).
\end{align}
On the other hand, since $c_2(\calt_\sigma(l)) = \deg(B_\nu)$ and $\calt_\sigma$ is reflexive of rank $2$, we have
$$
\deg(B_\nu) = c_2(\calt_\sigma(l)) = c_2(\calt_\sigma) + l \cdot c_1(\calt_\sigma) + l^2.
$$
From this, together with \ref{eq-1}, we obtain
$$
\deg(B_\nu) = l^2 - l(d_f + d_g) + (d_f^2 + d_g^2 + d_f d_g) -m(\sigma).
$$
\end{proof}

We define the Bourbaki degree of a sequence $\sigma$, inspired by the construction in \cite[Definition 2.4]{jardim2023bourbakidegreeplaneprojective}. 

\begin{defi}\label[defi]{def:Bour}
Let $\sigma = (f,g)$ be a normal sequence, $e = \indeg(\calt_\sigma)$ and let $d \doteq d_f + d_g$. The \emph{Bourbaki degree} of $\sigma$ is defined by:
$$
\Bour(\sigma) \doteq \deg(B_\nu) = e(e-d) + m_0 - m(\sigma),
$$
for some non-trivial syzygy $\nu \in H^0(\calt_\sigma(e))$. We note that every non-trivial syzygy of minimal degree is saturated, and $\Bour(\sigma)$ is independent of the choice of $\nu$.
\end{defi}

\begin{remark}\label[remark]{rmk:Bour-free-seq}
It follows from construction that the sequence $\sigma$ is free if and only if $\Bour(\sigma) = 0$, since every short exact sequence of the form
$$
0 \to \mathcal{O}_{\mathbb{P}^3}(-e) \to \calt_\sigma \to \mathcal{O}_{\mathbb{P}^3}(e-d) \to 0
$$
splits, as the extension group $
\ext^1(\mathcal{O}_{\mathbb{P}^3}(e-d), \mathcal{O}_{\mathbb{P}^3}(-e)) \simeq H^1(\mathcal{O}_{\mathbb{P}^3}(d-2e))$ vanishes.
\end{remark}

Whenever, up to a change of coordinates, the elements of a sequence $\sigma = (f,g)$ do not depend on some variable in $R$, $\sigma$ is called \emph{compressible}, see \cite{faenzi:hal-03271244}. The following lemma is a consequence of their results in our context:

\begin{Lemma}\label[Lemma]{lem:compressibility}[\cite[Lemmas 2.7 and 2.8]{faenzi:hal-03271244}]
Let $\sigma = (f,g)$ be a normal sequence in $\mathbb{P}^3$. Then:
\begin{itemize}
    \item[(a)] A sequence $\sigma$ is compressible if and only if $h^0(\calt_\sigma) \neq 0$;
    \item[(b)] If $\sigma$ is compressible, then $\calt_\sigma \simeq \mathcal{O}_{\mathbb{P}^3}\oplus \mathcal{O}_{\mathbb{P}^3}(-d)$.
\end{itemize}
\end{Lemma}

Now, we obtain some bounds for the quantities $\indeg(\calt_\sigma)$ and $m(\sigma)$ in terms of the degrees $d_f, d_g$.

\phantomsection\label{thm:first}
\begin{maintheorem}\label[maintheorem]{thmA}
\textit{ Let $\sigma = (f,g)$ be a normal sequence of homogeneous polynomials of the ring $\kappa[x_0, \ldots, x_3]$, with degrees $d_f + 1, d_g + 1$ respectively. Then:
\begin{itemize}
    \item[(a)] $\indeg(\calt_\sigma) \leq d_f + d_g$; 
    \item[(b)] $m(\sigma) \leq m_0$;
    \item[(c)] The following are equivalent:
    \begin{itemize}
        \item[(1)] $m(\sigma) = m_0$;
        \item[(2)] $\indeg(\calt_\sigma) = 0$;
        \item[(3)] $\sigma$ is compressible;
        \item[(4)] $\calt_\sigma \simeq \mathcal{O}_{\mathbb{P}^3} \oplus \mathcal{O}_{\mathbb{P}^3}(-d)$;
    \end{itemize}
\end{itemize}}
\end{maintheorem}

\begin{proof}
To show $(a)$, we build explicit syzygies of the Jacobian matrix $\nabla \sigma$ of degrees $d_f + d_g$, and at least one of them is nonzero. Writing the Jacobian matrix by
$$
 \nabla \sigma = \begin{pmatrix}
\partial_0 f & \partial_1 f & \partial_2 f & \partial_3 f\\
\partial_0 g & \partial_1 g & \partial_2 g & \partial_3 g\\
\end{pmatrix},
$$
the following vectors
\begin{align*}
\nu_0 &= \begin{pmatrix}
0\\
\partial_2 f \partial_3 g - \partial_3 f \partial_2 g \\
-\partial_1 f \partial_3 g + \partial_3 f \partial_1 g\\
\partial_1 f \partial_2 g - \partial_2 f \partial_1 g
\end{pmatrix};
\nu_1 = \begin{pmatrix}
\partial_2 f \partial_3 g - \partial_3 f \partial_2 g\\
0 \\
-\partial_0 f \partial_3 g + \partial_3 f \partial_0 g\\
\partial_0 f \partial_2 g - \partial_2 f \partial_0 g
\end{pmatrix};\\
\nu_2 &= \begin{pmatrix}
\partial_1 f \partial_3 g - \partial_3 f \partial_1 g\\
- \partial_0 f \partial_3 g + \partial_3 f \partial_0 g \\
0\\
\partial_0 f \partial_1 g - \partial_1 f \partial_0 g
\end{pmatrix};
\nu_3 = \begin{pmatrix}
\partial_1 f \partial_2 g - \partial_2 f \partial_1 g\\
- \partial_0 f \partial_2 g + \partial_2 f \partial_0 g \\
\partial_0 f \partial_1 g - \partial_1 f \partial_0 g\\
0
\end{pmatrix}.
\end{align*}
are a syzygies of degree $d_f + d_g$. Since $\sigma$ is algebraically independent, there is at least one nonzero $2\times 2$ minor, and at least one of the syzygies $\nu_0, \ldots, \nu_3$ is nonzero. 

For $(b)$, using that $e \leq d$, we can use that
$$
0 \leq \Bour(\sigma) = e(e-d) + m_0 - m(\sigma),
$$
so that $m(\sigma) \leq m_0 + e(e-d)$, but $e(e-d) \leq 0$ and the claim follows. 

For $(c)$, the equivalence $(2) \iff (3)$ is the content of \Cref{lem:compressibility}, $(a)$. Moreover, $(3) \Rightarrow (4)$ is \Cref{lem:compressibility}, $(b)$. The implication $(4) \Rightarrow (1)$ can be obtained using the Bourbaki degree formula, since $\indeg(\calt_\sigma) = 0$ and $\Bour(\sigma) = 0$, as $\sigma$ is free.

To show $(1) \Rightarrow (2)$, let $e = \indeg(\calt_\sigma)$. From the Bourbaki formula we obtain $\Bour(\sigma) = e(e-d)$. Since $\Bour(\sigma) \geq 0$ and, from part (a), $e(e-d) \leq 0$, it follows that either $e=0$ or $e=d$. Both cases mean $\Bour(\sigma) = 0$ and the sequence is free, say $\calt_\sigma \simeq \mathcal{O}_{\mathbb{P}^3}(-a) \oplus \mathcal{O}_{\mathbb{P}^3}(-b)$, with $0\leq a \leq b$. If $e = d$, then $a = d > 0$, but with $c_1(\calt_\sigma) = -(a+b)= -d$ one gets a contradiction with $a \leq b$, and thus $e = 0$.
\end{proof}

\begin{remark}
The bound obtained in \ref{thmA}, $(b)$ for $m(\sigma)$ relates to a known bound in foliation theory. We will use that $\calt_\sigma(1)$ is the tangent sheaf of a codimension-one foliation in $\mathbb{P}^3$ of degree $d = d_f + d_g$ (see \cite[Section 9]{faenzi:hal-03271244}). Let $C$ be the one-dimensional part of the singular scheme of this foliation. From the formulas of discrete invariants in \cite[Theorem 3.1]{CalvoAndrade-Jardim-Correa-Fol}:
$$
c_2(\calt_\sigma(1)) = d^2 + 2 - \deg(C).
$$
Using the formula $c_2(\calt_\sigma) = d_f^2 + d_g^2 + d_f d_g - m(\sigma)$ and the equations
$$
c_2(\calt_\sigma) - d + 1 = c_2(\calt_\sigma(1)) = d^2 + 2 - \deg(C)
$$
we obtain $
m(\sigma) = \deg(C) - d - d_f d_g - 1$, so from the bound above we get
\begin{align*}
\deg(C) - d - d_f d_g - 1 = m(\sigma)&\leq d_f^2 +d_g^2+ d_f d_g,
\end{align*}
and therefore $\deg(C) \leq d^2 + d + 1$, a bound that can be found in more generality for foliations in \cite[Corollary 4.8]{holomorphicfoliations-MSoares-2005}.
\end{remark}

The following inequalities relate the discrete invariants with $\mu$-stability and freeness of the logarithmic sheaves.

\begin{prop}\label[prop]{prop:Bour-lower-bound-stab}
Let $\sigma = (f,g)$ be a normal sequence of homogeneous polynomials in $R$. Denote by $e = \indeg(\calt_\sigma)$ and $d = d_f + d_g$. Then:
\begin{itemize}
    \item[(a)] $
    \Bour(\sigma) \leq m_0$, and equality holds if any of the equivalent facts hold:
    \begin{itemize}
        \item[(1)] $m(\sigma) = 0$;
        \item[(2)] $\calt_\sigma$ admits a minimal resolution by the Buchsbaum--Rim complex:
        $$
        0 \to \mathcal{O}_{\mathbb{P}^3}(-d-d_f) \oplus \mathcal{O}_{\mathbb{P}^3}(-d-d_g) \xrightarrow{\varphi} \mathcal{O}_{\mathbb{P}^3}^{\oplus 4}(-d) \xrightarrow{\psi} \calt_\sigma \to 0,
        $$
        where
        \[\psi=\begin{pmatrix}
		0 & \Delta_{12} & \Delta_{13} & \Delta_{14} \\
		-\Delta_{12} & 0 & \Delta_{23} & \Delta_{24} \\
		-\Delta_{13} & -\Delta_{23} & 0 & \Delta_{34} \\
		-\Delta_{14} & -\Delta_{24} & -\Delta_{34} & 0
	\end{pmatrix},\quad 
		\varphi=
		\begin{pmatrix}
			\partial_{0} f & -\partial_0 g \\[2pt]
			- \partial_1 f & \partial_1 g\\[2pt]
			\partial_2 f & -\partial_2 g\\[2pt]
			-\partial_3 f & \partial_3 g
		\end{pmatrix},
		\]
        and each $\Delta_{ij}$ is the $(i,j)-$minor of the matrix $\nabla \sigma$. In particular, $e = d$ and the logarithmic tangent sheaf is identified with the cokernel of the transpose matrix $\calt_\sigma(d) \simeq \coker(\nabla \sigma)^T$.
    \end{itemize}
    \item[(b)] If $
    \Bour(\sigma) > (d_f - 1)(d_f + d_g) + d_g^2 + 1$, then $\calt_\sigma$ is $\mu$-stable.
    \item[(c)] If
    $$
    m(\sigma) < \frac{1}{2}\left( \frac{3d_f^2}{2} + \frac{3d_g^2}{2} + d_f d_g \right),
    $$
    then $\sigma$ is not free.
    \end{itemize}
\end{prop}

\begin{proof}
The claim $(a)$ follows from the formula $\Bour(\sigma) = e(e-d) + m_0 - m(\sigma)$, since $m(\sigma) \geq 0$ and $e(e-d) \leq 0$. This also shows equality occurs if and only if both $e(e-d)$ and $m(\sigma)$ are zero, therefore $e = d$ or $e = 0$. But $e = 0$ gives $\Bour(\sigma) = 0$, from compressibility, and $m_0 \neq 0$ implies $e = d$. To show the equivalences above, note that $m(\sigma) = 0$ if and only if the grade of the ideal of $(2 \times 2)-$minors is three, which occurs if and only if the Buchsbaum--Rim complex is a minimal free resolution for $\nabla \sigma$, when in particular $e = d$. 

To show $(b)$, note that
$$
\Bour(\sigma) \leq e(e-d) + m_0,
$$
and this is a function $H = H(e)$ which attains its minimum at $e = d/2$. The function $H$ is decreasing on $e \in \{1, \ldots, d/2\}$, so when $\Bour(\sigma) > H(1)=(d_f - 1)(d_f + d_g) + d_g^2 + 1$, then $e > d/2$, and thus $\calt_\sigma$ is $\mu$-stable.

To prove $(c)$, note that $H$ attains its minimum at $e = d/2$, so we obtain that
$$
\Bour(\sigma) = H(e) - m(\sigma) \geq H(d/2) - m(\sigma) = \frac{d^2}{4} - \frac{d^2}{2} + m_0 - m(\sigma) > 0,
$$
and therefore $\Bour(\sigma) \neq 0$, independently of the value of $e = \indeg(\calt_\sigma)$, and $\sigma$ is not free.
\end{proof}

\begin{remark}\label[remark]{rmk:regular-pencils}
Let $\sigma = (f,g)$ with $d_f = d_g = p$. It defines a pencil of projective surfaces:
$$
V_{\sigma} \doteq \{ V(z_0 f+z_1 g) \subset \mathbb{P}^3 : z = [z_0:z_1] \in \mathbb{P}^1 \},
$$
where each $V(z_0 f + z_1 g)$ is a \textit{member} of $V_\sigma$. The pencil is called \emph{regular} when the generic member is non-singular, otherwise it is called \emph{irregular}.

In \cite[Lemma 2.17]{faenzi:hal-03271244}, the authors describe the support of the sheaf $\mathcal{Q}_\sigma = \coker(\nabla \sigma)$ for pencils of surfaces as the union of singular loci:
$$
(\Xi_\sigma)_{\text{red}} = \bigcup_{[z_0:z_1] \in \mathbb{P}^1} \Sing(V(z_0 f + z_1 g)).
$$
Hence, whenever a normal sequence $\sigma=(f,g)$ defines a regular pencil whose members have at most isolated singularities, then $\dim(\Xi_\sigma) = 0$ and $m(\sigma) = 0$.
\end{remark}

\begin{example}\label[example]{ex:pencils-cubics-genericpencil}
Consider the sequence
$$
\sigma = (f,g) = ( x_3(x_0x_2 - x_1^2) - (x_0 - 2x_1)(3x_1 - x_0- 2x_2)(x_1 - 2x_2), x_3(x_0x_2 - x_1^2) - x_1^2(x_0-x_1))
$$
where $f$ is a normal singular cubic with an $A_1$-singularity at $[0:0:0:1]$ and $g$ is a normal singular cubic with singularity type $2A_1A_2$. Here, $m(\sigma) = 0$, $\Bour(\sigma) = 12$ and $c_3(\calt_\sigma) = 32$. Using \Cref{prop:Bour-lower-bound-stab}, $(a)$, it follows $\calt_\sigma$ has a minimal free resolution given by the Buchsbaum--Rim complex:
$$
0 \rightarrow \mathcal{O}_{\mathbb{P}^3}(-6)^{\oplus 2} \rightarrow \mathcal{O}_{\mathbb{P}^3}(-4)^{\oplus 4} \rightarrow \calt_\sigma \rightarrow 0.
$$
\end{example}

\subsection{Plane curves as pairs of surfaces}\label[subsection]{subsec:bourbaki-degree-proj-plane-curve}

Let $g \in \kappa[x_0, x_1, x_2]$ be a square-free homogeneous polynomial of degree $d$, with the associated curve $X = V(g) \subset \mathbb{P}^2 = V(x_3)$ which has at most isolated singularities, and let $\sigma = (x_3, g)$. We consider $S = V(g) \subset \mathbb{P}^3$, which is the cone over the curve $X$ and denote $Z_g \doteq \Sing(S)$ the associated singular scheme. The matrix $\nabla \sigma$ will be given by
$$
\nabla \sigma = \begin{pmatrix}
0 & 0 & 0 & 1\\
\partial_0 g & \partial_1 g & \partial_2 g & 0
\end{pmatrix}.
$$
Let us denote by $\nabla \overline{g} = (\partial_0 g, \partial_1 g, \partial_2 g)$ and
$$
\calt_{\overline{g}} \doteq \ker(\nabla \overline{g}) \hookrightarrow \mathcal{O}_{\mathbb{P}^3}^{\oplus 3} \xrightarrow{\nabla \overline{g}} \mathcal{O}_{\mathbb{P}^3}(d_g)
$$
the kernel of the multiplication. Using the block-form of the matrix $\nabla \sigma$, we form the following diagram with exact columns:
\[\begin{tikzcd}
	{\calt_{\overline{g}}} & {\mathcal{O}_{\mathbb{P}^3}^{\oplus 3}} & {\mathcal{O}_{\mathbb{P}^3}(d)} & {\mathcal{O}_{Z_g}(d)} \\
	{\calt_\sigma} & {\mathcal{O}_{\mathbb{P}^3}^{\oplus 4}} & {\mathcal{O}_{\mathbb{P}^3}(d) \oplus \mathcal{O}_{\mathbb{P}^3}} & {\mathcal{Q}_\sigma} \\
	& {\mathcal{O}_{\mathbb{P}^3}} & {\mathcal{O}_{\mathbb{P}^3}}
	\arrow[hook, from=1-1, to=1-2]
	\arrow["{\nabla \overline{g}}", from=1-2, to=1-3]
	\arrow[hook, from=1-2, to=2-2]
	\arrow[two heads, from=1-3, to=1-4]
	\arrow[hook, from=1-3, to=2-3]
	\arrow[hook, from=2-1, to=2-2]
	\arrow["{\nabla \sigma}", from=2-2, to=2-3]
	\arrow[two heads, from=2-2, to=3-2]
	\arrow[two heads, from=2-3, to=2-4]
	\arrow[two heads, from=2-3, to=3-3]
	\arrow["{\cdot 1}", from=3-2, to=3-3]
\end{tikzcd}\]
and by the snake lemma, there are isomorphisms $\calt_\sigma \simeq \calt_{\overline{g}}$ and $\mathcal{Q}_\sigma \simeq \mathcal{O}_{Z_g}(d)$. On $\mathbb{P}^2$, we have the sequence defining the logarithmic tangent sheaf of the curve $X$:
$$
0 \to \calt_g \to \mathcal{O}_{\mathbb{P}^2} \xrightarrow{(\partial_0 g, \partial_1 g, \partial_2 g)} \mathcal{O}_{\mathbb{P}^3}(d) \to 0,
$$
and a defined Bourbaki degree $\Bour(X)$ (\cite{jardim2023bourbakidegreeplaneprojective}), which coincides with the Chern class $c_2(\calt_g(e))$ when $e = \indeg(\calt_g)$. By the block-form of $\nabla \sigma$,
$$
e = \indeg(\calt_g) = \indeg(\calt_\sigma).
$$
Denoting by $i: \mathbb{P}^2 \simeq V(x_3) \hookrightarrow \mathbb{P}^3$ the inclusion, it is clear that $i^*(\calt_\sigma) \simeq \calt_{g}$ as a logarithmic sheaf over $\mathbb{P}^2$, since $H$ and $Z_g$ intersect transversely. Now, compare the following formulas for the Bourbaki degrees of $\sigma$ and of $X = V(g) \subset \mathbb{P}^2$:
$$
\begin{cases}
c_2(\calt_\sigma(e))&= \Bour(\sigma) = e(e-d_g) + d_g^2 - m(\sigma)\\
c_2(\calt_{g}(e))&= \Bour(X)      = e(e-d_g) + d_g^2 - \tau(X),
\end{cases}
$$
where $\tau(X)$ is the Tjurina number of the curve $X$. From the transversality, we obtain 
$i^*(c(\calt_\sigma)) = c(i^*(\calt_\sigma))$, so that
\begin{align*}
1 - d[H] + (m_0 - m(\sigma))[H]^2 &= i^*(c(\calt_\sigma)) = c(i^*(\calt_\sigma))\\
&= 1 - d[H] + (d^2 - \tau(X))[H]^2.
\end{align*}
Since $m_0 = d_g^2$, we obtain the equality $m(\sigma) = \tau(X)$. Geometrically, this means that $m(\sigma)$ counts the singular lines of the cone $S = V(g) \subset \mathbb{P}^3$ with the same multiplicity as the Tjurina number does. Another perspective for this is through the associativity formula. Note that $\mathcal{Q}_\sigma \otimes \mathcal{O}_{H} \simeq \mathcal{O}_{\Sing(X)}(d)$ or, in terms of rings, if $S \doteq \kappa[x_0, x_1, x_2]$, we have
$$
\frac{S[x_3]}{(\partial_0 g, \partial_1 g, \partial_2 g)} \otimes_S \frac{S[x_3]}{(x_3)} \simeq \frac{S}{(\partial_0 g, \partial_1 g, \partial_2 g)},
$$
and then coming back to $R$ via a flat change of basis $\otimes_{\kappa} \kappa[x_3]$, since $\partial_i g \in \kappa[x_0, x_1, x_2]$. Therefore, for any associated prime $\fp$ of the Jacobian ideal $J_g \subset S$,
the extension satisfies
\begin{align*}
\dim R/(\fp R) &= \dim(S/\fp) + 1\\
\deg R/(\fp R) &= \deg(S/\fp),
\end{align*}
and each associated prime to the cone ideal $J_g R$ comes by an extension to an associated prime of $J_g$ in $S$.
Now, using the associativity formula, we get:
\begin{align*}
\tau(X) &= \sum_{\fp \in \Ass(J_g), \dim(S/\fp) = 1} \length_{S_\fp}(S/J_g)_\fp \cdot \deg(S/\fp)\\
&= \sum_{\fp \in \Ass(J_g R), \dim(R/\fp) = 2} \length_{R_\fp}(R/J_g R)_\fp \cdot \deg(R/\fp) = m(\sigma).
\end{align*}

The following example is a sequence $\sigma = (x_3,g)$ of degree $d$ such that $e = d$ is maximal but $m(\sigma) \neq 0$.

\begin{example}\label[example]{ex:maximal-e-m-not-m0}
Let $f = x_3$ and $g = x_0x_1x_2^2 + x_0^4 + x_1^4$. The quartic plane curve $X = V(g) \subset \mathbb{P}^3_{[x_0:x_1:x_2]}$ is nodal and the minimum degree for a syzygy of $\nabla g$ is $e = d$.

Hence, $m(\sigma) = \tau(X) = 1 \neq 0$, accounting for the nodal singularity, but $e = d = 3$.
\end{example}

\subsection{Free resolutions and Bourbaki schemes}\label[subsection]{sec:resolutions}

In this section, after a choice of syzygy $\nu$ of minimal degree, we relate the resolutions for the ideal sheaf of its zero-locus $\mathcal{I}_{B_\nu}$ and for $\calt_\sigma$. We use this lemma to characterize different classes of sequences.

\begin{Lemma}\label[Lemma]{prop:lifting-resolutions}
Let $\nu \in H^0(\calt_\sigma(e))$ be a nonzero section with $e = \indeg(\calt_\sigma)$ and let $B \subset \mathbb{P}^3$ be the pure codimension $2$ subscheme associated with $\nu$ in a short exact sequence:
$$
0 \rightarrow \mathcal{O}_{\mathbb{P}^3}(-e) \xrightarrow{\nu} \calt_\sigma \xrightarrow{\pi} \mathcal{I}_B(e-d) \rightarrow 0.
$$
Then:
\begin{itemize}
    \item[(a)] Every free resolution for $\mathcal{I}_B$:
$$
0 \rightarrow F_2 \rightarrow F_1 \rightarrow F_0 \xrightarrow{\omega} \mathcal{I}_B \rightarrow 0
$$
lifts to a free resolution
$$
0 \rightarrow F_2(e-d) \rightarrow F_1(e-d) \rightarrow F_0(e-d) \oplus \mathcal{O}_{\mathbb{P}^3}(-e) \xrightarrow{(\omega(e-d), \nu)} \calt_\sigma \rightarrow 0.
$$
\item[(b)] Every minimal free resolution of $\calt_\sigma$ including the section $\nu$:
$$
0 \rightarrow F_2 \rightarrow F_1 \rightarrow F_0 \oplus \mathcal{O}_{\mathbb{P}^3}(-e) \xrightarrow{(\lambda, \nu)} \calt_\sigma \rightarrow 0,
$$
induces a free resolution of $\mathcal{I}_B$:
$$
0 \rightarrow F_2(d-e) \rightarrow F_1(d-e) \rightarrow{} F_0(d-e) \xrightarrow{\lambda(d-e)} \mathcal{I}_B \rightarrow 0.
$$
\end{itemize}
\end{Lemma}

\begin{proof}
To summarize, item $(a)$ follows from the Horseshoe lemma since line bundles have vanishing $\ext^1(-, \mathcal{O}_{\mathbb{P}^3})$, and $(b)$ comes from the mapping cone, after factoring the trivial factors $\mathcal{O}_{\mathbb{P}^3}(-e)$.

To show $(a)$, we apply the functor $\Mor(F_0(e-d), -)$ to the short exact sequence
$$
0 \rightarrow \mathcal{O}_{\mathbb{P}^3}(-e) \xrightarrow{\nu} \calt_\sigma \xrightarrow{\pi} \mathcal{I}_B(e-d) \rightarrow 0,
$$
to get the exact piece:
$$
\Mor(F_0(e-d), \calt_\sigma) \xrightarrow{\pi^*} \Mor(F_0(e-d), \mathcal{I}_B(e-d)) \rightarrow \ext^1(F_0(e-d), \mathcal{O}_{\mathbb{P}^3}(-e)) = 0,
$$
since $
\ext^1(F_0(e-d), \mathcal{O}_{\mathbb{P}^3}(-e)) \simeq H^1(F_0^\vee(-2e-d)) = 0$, as $F_0^\vee$ is a direct sum of line bundles and these have vanishing first cohomology in $\mathbb{P}^3$. Thus, $\pi^*$ is surjective, and there is a morphism $\Tilde{\omega}: F_0(e-d) \rightarrow \calt_\sigma$ such that $\pi \circ \Tilde{\omega} = \omega(e-d)$. We now consider the map $\omega(e-d) \oplus \nu$ in the following commutative diagram with short exact sequences as the central two columns:
\[\begin{tikzcd}
	& {\mathcal{O}_{\mathbb{P}^3}(-e)} & {\mathcal{O}_{\mathbb{P}^3}(-e)} \\
	{\ker(\Tilde{\omega} \oplus \nu)} & {F_0(e-d) \oplus \mathcal{O}_{\mathbb{P}^3}(-e)} & {\calt_\sigma} & {\coker(\Tilde{\omega} \oplus \nu)} \\
	{\ker(\omega(e-d))} & {F_0(e-d)} & {\mathcal{I}_B(e-d)} & 0
	\arrow[Rightarrow, no head, from=1-2, to=1-3]
	\arrow[hook, from=1-2, to=2-2]
	\arrow["\nu", hook, from=1-3, to=2-3]
	\arrow[hook, from=2-1, to=2-2]
	\arrow["{\Tilde{\omega} \oplus \nu}", from=2-2, to=2-3]
	\arrow[two heads, from=2-2, to=3-2]
	\arrow[two heads, from=2-3, to=2-4]
	\arrow["\pi", two heads, from=2-3, to=3-3]
	\arrow[hook, from=3-1, to=3-2]
	\arrow["{\omega(e-d)}", from=3-2, to=3-3]
	\arrow[from=3-3, to=3-4]
\end{tikzcd}\]
From the snake lemma, we obtain that $\coker(\Tilde{\omega} \oplus \nu) = 0$ and that $\ker(\Tilde{\omega} \oplus \nu) \simeq \ker(\omega(e-d))$. Thus, we can continue the resolution for $\mathcal{I}_B$, twisting by $\mathcal{O}_{\mathbb{P}^3}(e-d)$, to obtain the following free resolution:
\[\begin{tikzcd}
	0 & {F_2(e-d)} & {F_1(e-d)} & {F_0(e-d)\oplus \mathcal{O}_{\mathbb{P}^3}(-e)} & {\calt_\sigma} & 0 \\
	&&& {\ker(\omega(e-d))}
	\arrow[from=1-1, to=1-2]
	\arrow[from=1-2, to=1-3]
	\arrow[from=1-3, to=1-4]
	\arrow[two heads, from=1-3, to=2-4]
	\arrow["{\Tilde{\omega}\oplus\nu}", from=1-4, to=1-5]
	\arrow[from=1-5, to=1-6]
	\arrow[hook, from=2-4, to=1-4]
\end{tikzcd}\]
for $\calt_\sigma$, as claimed.

To show $(b)$, we consider the diagram with exact rows induced by the fact above to obtain the short exact sequence in cokernels as the third row below:
\[\begin{tikzcd}
	& {\mathcal{O}_{\mathbb{P}^3}(-e)} & {\mathcal{O}_{\mathbb{P}^3}(-e)} \\
	S & {F_0\oplus \mathcal{O}_{\mathbb{P}^3}(-e)} & {\calt_\sigma} \\
	S & {F_0} & {\mathcal{I}_B(e-d)}
	\arrow[Rightarrow, no head, from=1-2, to=1-3]
    \arrow["\nu", hook, from=1-2, to=2-2]
	\arrow["\nu", hook, from=1-3, to=2-3]
	\arrow[hook, from=2-1, to=2-2]
	\arrow[Rightarrow, no head, from=2-1, to=3-1]
	\arrow[two heads, from=2-2, to=2-3]
	\arrow[two heads, from=2-2, to=3-2]
	\arrow["\pi", two heads, from=2-3, to=3-3]
	\arrow[hook, from=3-1, to=3-2]
	\arrow[two heads, from=3-2, to=3-3]
\end{tikzcd}\]
Completing to the resolution and twisting accordingly, we obtain the resolution from the claim below.
\[\begin{tikzcd}
	0 & {F_2(d-e)} & {F_1(d-e)} & {F_0(d-e)} & {\mathcal{I}_B} & 0 \\
	&&& {S(d-e)}
	\arrow[from=1-1, to=1-2]
	\arrow[from=1-2, to=1-3]
	\arrow[from=1-3, to=1-4]
	\arrow[two heads, from=1-3, to=2-4]
	\arrow[from=1-4, to=1-5]
	\arrow[from=1-5, to=1-6]
	\arrow[hook, from=2-4, to=1-4]
\end{tikzcd}\]
\end{proof}

\begin{defi}\label[defi]{def:geometry-of-B}
Let $\sigma$ be a non-free normal sequence with degrees $d_f + 1, d_g + 1$. We say that $\sigma$ is:
\begin{itemize}
    \item \emph{nearly free} if $\Bour(\sigma) = 1$.
    \item \emph{$3$-syzygy} if there is a minimal free resolution for $\calt_\sigma$ such that $\rank(F_0) = 2$ in the notation of \Cref{prop:lifting-resolutions}, $(b)$.
\end{itemize}
\end{defi}

\begin{example}[Nearly-free sequence with $d_f = 1, d_g = 2$]\label[example]{ex:nearly-free-mixed-degrees-e1}
We consider the following sequence with $d_f = 1$, $d_g = 2$:
$$
\sigma = (x_0^2 + x_3^2, x_0^3 + x_0x_1x_2+x_3^3)
$$
with Jacobian matrix given by
$$
\nabla \sigma = \begin{pmatrix}
2x_0 & 0 & 0 & 2x_3\\
3x_0^2 +x_1x_2 & x_0x_2 & x_0x_1 & 3x_3^2
\end{pmatrix}.
$$
The scheme $(\Xi_\sigma)_1$ has three primary components in dimension one, given by the prime ideals $\fp_1 = (x_0, x_3)$, $\fp_2 = (x_0, x_2)$ and $\fp_3 = (x_0, x_1)$. We also note that $\nu = (-x_3, 0, 0, x_0)^T$ is a syzygy for $\nabla \sigma$ of degree one, so $e \leq 1$. We claim that $m(\sigma) = 4$, which shows that $\sigma$ cannot be compressible (by \Cref{thmA}) and hence $e = 1$. Moreover, with these data, the Bourbaki degree formula yields $\Bour(\sigma) = 1$, so this is a nearly-free sequence. To show $m(\sigma) = 4$, we study the length of the cokernel module $Q$ associated with the sheaf $\mathcal{Q}_\sigma$ over each prime $\fp_i$, $i =1, 2, 3$, and use the associativity formula.

Over $\fp_1 = (x_0, x_3)$, we note that $x_1 x_2 \in R_\fp^\times$ is a unit, and $3x_0^2 \in \fp$ is in the maximal ideal, hence $v = 3x_0^2 + x_1 x_2$ is also invertible, denote $u = v^{-1}$. From this, we may rewrite the matrix as
$$
(\nabla \sigma)_{\fp_1} = \begin{pmatrix}
2x_0 & 0 & 0 & 2x_3\\
1    & ux_0x_2 & ux_0x_1 & 3ux_3^2\\
\end{pmatrix} \sim \begin{pmatrix}
0    & -2x_0^2x_2u             & -2ux_0^2x_1 & 2x_3-6ux_0x_3^2\\
1    & x_0x_2 u                & x_0x_1 u    & 3x_3^2 u\\
\end{pmatrix},
$$
first by multiplying the second row by $u$, and then adding the second row scaled by $-2x_0$ to the first row. From the final form of the matrix above, it is easy to see that it sends the fourth basis vector to $(0,1) \in R_{\fp_1}^2$, and thus we may compute the cokernel as the remaining entries of the first row:
$$
\coker(\nabla\sigma)_{\fp_1} \simeq \frac{R_{\fp_1}}{(-2x_0^2x_2u, -2ux_0^2x_1, 2x_3-6ux_0x_3^2)} \simeq \frac{R_{\fp_1}}{(x_0^2, x_3)},
$$
a $R_{\fp_1}-$module of length two.

Over the prime $\fp_2 = (x_0,x_2)$, the element $x_3 \in R_{\fp_2}$ is an unit. We denote by $u = x_3^{-1} \in R_{\fp_2}$ its inverse, and we rewrite the matrix using elementary operations
$$
(\nabla \sigma)_{\fp_2} = \begin{pmatrix}
2x_0                     & 0 & 0 & 2x_3\\
3x_0^2+x_1x_2-3x_0x_3    & x_0x_2 & x_0x_1 & 0\\
\end{pmatrix} \sim \begin{pmatrix}
u x_0                     & 0 & 0           & 1\\
3x_0^2+x_1x_2-3x_0x_3    & x_0x_2 & x_0x_1 & 0\\
\end{pmatrix},
$$
first by adding the first row scaled by $(-\frac{3}{2}x_3)$ to the second row, and then multiplying the first row by $u/2$. Now, it is clear that the matrix sends the fourth basis vector to $(1,0) \in R_\fp^{2}$, and thus
\begin{align*}
\coker(\nabla \sigma)_{\fp_2} &\simeq \frac{R_{\fp_2}}{(x_0x_1, x_0x_2, 3x_0^2 + x_1 x_2 - x_0x_3)}\\
&\simeq \frac{R_{\fp_2}}{(x_0, x_0x_2, 3x_0^2 + x_1 x_2 - x_0x_3)}\\
&\simeq \frac{R_{\fp_2}}{(x_0, x_0x_2, x_1 x_2)}\\
&\simeq \frac{R_{\fp_2}}{(x_0, x_2)} \simeq \kappa
\end{align*}
since $x_1 \in R_{\fp_2}^\times$. Thus, it follows $\length(\mathcal{Q}_\sigma)_{\fp_2} = 1$. For the ideal $\fp_3$, we note that the approach is analogous as the previous one, using the same elementary operations, and obtaining the same isomorphism at the end $\coker(\nabla \sigma)_{\fp_3} \simeq \kappa$, so that from the associativity formula:
$$
m(\sigma) = \sum_{i=1}^3 \length((\mathcal{Q}_\sigma)_{\fp_i}) \cdot \deg(R/\fp_i) = 2 + 1 + 1 = 4,
$$
as we wanted to show.
\end{example}

\begin{example}[$3$-syzygy pencil of cubics which is not nearly-free]\label[example]{ex:pcubics-pog-not-nf}
Consider the following pencil of cubics:
$$
\sigma = (f,g) =(x_2x_3(x_0 - x_1), x_0(x_0^2 + x_1^2 + x_2^2 + x_3^2))
$$
where $f$ is a hyperplane arrangement and $g$ is the union of a plane and a smooth quadric, with the Jacobian matrix:
$$
\nabla \sigma = \begin{pmatrix}
x_2x_3 & -x_2x_3 & x_3(x_0-x_1) & x_2(x_0-x_1)\\
3x_0^2+x_1^2+x_2^2+x_3^2 & 2x_0x_1 & 2x_0x_2 & 2x_0x_3
\end{pmatrix}.
$$
We claim that here, $m(\sigma) = 5$ and that $\Bour(\sigma) = 4 \neq 1$, so $\sigma$ is not nearly free. Macaulay2 gives the following minimal free resolution for $\calt_\sigma$:
$$
0 \rightarrow \mathcal{O}_{\mathbb{P}^3}(-5) \rightarrow \mathcal{O}_{\mathbb{P}^3}(-3)^{\oplus 3} \rightarrow \calt_\sigma \rightarrow 0,
$$
which is of $3$-syzygy type. The irreducible components of $(\Xi_\sigma)_1$ are three lines and a plane quadric, given by the following prime ideals in $R$:
\begin{align*}
\fp_1 &= (x_2, x_3)\\
\fp_2 &= (x_3, x_0-x_1)\\
\fp_3 &= (x_2, x_0-x_1)\\
\fp_4 &= (x_0, x_0^2 + x_1^2 + x_2^2 + x_3^2)
\end{align*}
For $\fp_1 = (x_2, x_3)$, we note that the elements
\begin{align*}
h &= 3x_0^2 + x_1^2 + x_2^2 + x_3^2\\
v &= x_0x_1\\
t &= (x_0-x_1)
\end{align*}
are all invertible in the local ring $R_{\fp_1}$, where the first one is the sum of an invertible element with an element inside the maximal ideal $(x_2, x_3)_{\fp_1}$. Then, we do two elementary operations:
\begin{align*}
(\nabla \sigma)_{\fp_1} &\sim \begin{pmatrix}
ux_2 x_3 + \frac{1}{2} h x_2 x_3 & 0 & utx_3 + x_0x_2^2x_3 & ut x_2 + x_0x_2x_3^2\\
h & 2v & 2x_0x_2 & 2x_0x_3
\end{pmatrix}\\
&\sim \begin{pmatrix}
ux_2 x_3 + \frac{1}{2} h x_2 x_3 & 0 & utx_3 + x_0x_2^2x_3 & ut x_2 + x_0x_2x_3^2\\
\frac{1}{2}uh & 1 & ux_0x_2 & ux_0x_3
\end{pmatrix}.
\end{align*}
The first is to add $\frac{1}{2}x_2 x_3$ times the second row to the first row, scaled by $u = v^{-1}$. Then, we scale the second row by $\frac{1}{2}u$. Since the matrix sends the second basis vector into $(0,1) \in R_{\fp_1}^2$, it suffices to look at the first row of the remaining matrix, so that
\begin{align*}
\coker(\mathcal{Q}_\sigma)_{\fp_1} &\simeq \frac{R_{\fp_1}}{((u + \frac{1}{2}h)x_2x_3, utx_3+x_0x_2(x_2x_3), ut x_2 + x_0x_3(x_2x_3) )}\\
&\simeq \frac{R_{\fp_1}}{(x_2x_3, utx_3+x_0x_2(x_2x_3), ut x_2 + x_0x_3(x_2x_3) )}\\
&\simeq \frac{R_{\fp_1}}{(x_2x_3, utx_3, ut x_2)} \simeq \frac{R_{\fp_1}}{(x_2, x_3)} \simeq \kappa,
\end{align*}
and hence $\length((\mathcal{Q}_\sigma)_{\fp_1}) = 1$. 

For $\fp_2 = (x_3, x_0-x_1)$, note that $v = x_0x_2 \in R_{\fp_2}$ is a unit and denote $u = v^{-1}$. We consider some elementary operations:
\begin{align*}
(\nabla \sigma)_{\fp_2} &\sim \begin{pmatrix}
x_2 x_3 - \frac{1}{2} u x_3(x_0-x_1)h & -x_2x_3 -\frac{1}{2} u x_3(x_0-x_1)x_0x_1  & 0 & x_2(x_0-x_1) - \frac{1}{2} u x_3(x_0-x_1)x_0x_3\\
h & 2x_0x_1 & 2v & 2x_0x_3
\end{pmatrix}\\
&\sim \begin{pmatrix}
x_2 x_3 - \frac{1}{2} u x_3(x_0-x_1)h & -x_2x_3 -\frac{1}{2} u x_3(x_0-x_1)x_0x_1  & 0 & x_2(x_0-x_1) - \frac{1}{2} u x_3(x_0-x_1)x_0x_3\\
\frac{1}{2}uh & ux_0x_1 & 1 & ux_0x_3
\end{pmatrix}.
\end{align*}
The first transformation is to add the scaled second row (by $-\frac{1}{2}x_3(x_0-x_1)u$) to the first one, and the second transformation is rescaling the second row by $\frac{1}{2}u$. Then, we see that the matrix sends the third basis vector to $(0,1) \in R_{\fp_2}^2$, so we may compute the cokernel as the first row of the remaining matrix:
\begin{align*}
\coker(\mathcal{Q}_\sigma)_{\fp_2} &\simeq \frac{R_{\fp_2}}{(x_3(x_2 - \frac{1}{2}u(x_0-x_1)h), -x_3(x_2+\frac{1}{2}u(x_0-x_1)x_0x_1), (x_0-x_1)(x_2-\frac{1}{2}ux_0x_3^2)))}\\
&\simeq \frac{R_{\fp_2}}{x_3, -x_3, (x_0-x_1)} \simeq \frac{R_{\fp_2}}{(x_3, x_0-x_1)} \simeq \kappa,
\end{align*}
since for example the quantity
$$
x_2 - \frac{1}{2}u(x_0-x_1)h
$$
is of the form a unit $x_2$ plus something inside the maximal ideal $(x_3, x_0-x_1)_{\fp_2}$, it is also a unit. Thus, we obtain $\length((\mathcal{Q}_\sigma)_{\fp_2}) = 1$. An analogous consideration also holds for the case of the prime $\fp_3 = (x_2, x_0-x_1)$, with similar elementary operations, now to trivialize the fourth column as $(0,1)$, and then to show that $(\mathcal{Q}_\sigma)_{\fp_3} \simeq \kappa$.

For the prime $\fp_4 = (x_0, h \doteq x_0^2 + x_1^2 + x_2^2 + x_3^2)$, we note that $v = x_2x_3$ is invertible in $R_{\fp_4}$, denote $u = v^{-1}$ and do the elementary operations
\begin{align*}
(\nabla \sigma)_{\fp_4} &\sim \begin{pmatrix}
v            & -v    & x_3(x_0-x_1) & x_2(x_0-x_1)\\
h+2x_0x_1    &  0    & 2x_0x_2+2x_0x_1x_3(x_0-x_1) & 2x_0x_2x_3(x_0-x_1)\\
\end{pmatrix}\\
&\sim \begin{pmatrix}
-1            & 1    & -ux_3(x_0-x_1) & -ux_2(x_0-x_1)\\
h+2x_0x_1    &  0    & 2x_0x_2+2x_0x_1x_3(x_0-x_1) & 2x_0x_2x_3(x_0-x_1)\\
\end{pmatrix},
\end{align*}
first adding the scaled first row $(by 2x_0x_1u)$ to the second row, and then scaling the first row by $-u$. The matrix sends the second basis vector to $(1,0) \in R_{\fp_4}^{2}$, and thus to compute the cokernel we may consider only the second row of the remaining matrix, so that
\begin{align*}
(\mathcal{Q}_\sigma)_{\fp_4} &\simeq \frac{R_{\fp_4}}{( h+2x_0x_1, x_0(2x_2 + 2x_0x_1x_3(x_0-x_1)), 2x_0x_2x_3(x_0-x_1) )}\\
&\simeq \frac{R_{\fp_4}}{( h+2x_0x_1, x_0,x_0)}\\
&\simeq \frac{R_{\fp_4}}{(x_0, h)} \simeq \kappa,
\end{align*}
noting that the factors $(2x_2 + 2x_0x_1x_3(x_0-x_1))$ and $2x_2x_3(x_0-x_1)$ are units in the local ring. Thus, we obtain that the length at each prime is one, but one of them has multiplicity two, therefore $m(\sigma) = 5$ from the associativity formula.

Using \Cref{prop:lifting-resolutions} with the minimal free resolution for $\calt_\sigma$ above, after a choice of minimal degree syzygy, one gets the following resolution for $B = B_\nu$:
$$
0 \rightarrow \mathcal{O}_{\mathbb{P}^3}(-4) \rightarrow \mathcal{O}_{\mathbb{P}^3}(-2)^{\oplus 2} \rightarrow \mathcal{I}_B \rightarrow 0,
$$
which presents $B$ as a complete intersection of two quadric surfaces in $\mathbb{P}^3$.
\end{example}

For the rest of the section, we study some aspects of these special classes of sequences.

\begin{prop}\label[prop]{prop:nearly-free-seq}
Let $\sigma = (f,g)$ be a normal sequence with degrees $d_f + 1$, $d_g + 1$. Then $\sigma$ is nearly free if and only if the sheaf $\calt_\sigma$ admits a free resolution  of the form:
$$
0 \rightarrow \mathcal{O}_{\mathbb{P}^3}(e-d-2) \rightarrow \mathcal{O}_{\mathbb{P}^3}(e-d-1)^{\oplus 2} \oplus \mathcal{O}_{\mathbb{P}^3}(-e) \rightarrow \calt_\sigma \rightarrow 0,
$$
where $d = d_f + d_g$ and $e = \indeg(\calt_\sigma)$. Moreover, when $\sigma$ is nearly free, the zero-dimensional part of $\Xi_\sigma$ is aligned.
\end{prop}

\begin{proof}
Note that $\Bour(\sigma) = 1$ if and only if $B = B_\nu$ is a line, for a given choice of non-zero syzygy of minimal degree $\nu$. Then, the equivalence follows from  \Cref{prop:lifting-resolutions}. In general, the singular scheme of $\calt_\sigma$ is supported at the zero-dimensional locus of $\Xi$, and this will be contained inside $B$, a fact which may be seen from the short exact sequence
$$
0 \to \mathcal{O}(-e) \to \calt_\sigma \to \mathcal{I}_B(e-d) \to 0.
$$
\end{proof}

\begin{remark}\label[remark]{rmk:nearly-free-curves}
The notion of \emph{nearly-free curves} for plane curves $V(f) \subset \mathbb{P}^2$ is first introduced by \cite{Dimca-Sticlaru-freenearlyfree2018}, related to rational cuspidal curves. In \cite[Proposition 2.18]{jardim2023bourbakidegreeplaneprojective}, the authors show that $\Bour(f) = 1$ if and only if $V(f) \subset \mathbb{P}^2$ is a nearly free curve in the sense of \cite{Dimca-Sticlaru-freenearlyfree2018} (see \cite[Definition 2.17]{jardim2023bourbakidegreeplaneprojective}). The notion of $3$-syzygy divisors is also present in a number of previous works, for example \cite{TakuroAbe-plusonegenerated-2021}, \cite{dimca2020plane} and \cite{DimcaSticlaru-2025-plusonegeneratedcurves}.
\end{remark}

\begin{example}\label[example]{ex:free-nearlyfree-sequence-of-examples}
Consider $f = x_0^3 + x_0x_1x_2 + x_3^3$ and $g = x_0^{k+1} + x_3^{k+1}$ for $k \geq 2$. Then $d = d_f + d_g = k+2$ and:
\begin{itemize}
    \item For $k = 2$, $(f,g)$ is a free pencil of cubics with $e = 1$, $m(\sigma) = 9$;
    \item For $k > 2$, $(f,g)$ is nearly free with $e = 1$.
\end{itemize}
The Jacobian matrix for $\sigma = (f,g)$ is
$$
\nabla \sigma = \begin{pmatrix}
3x_0^2 + x_1 x_2 & x_0x_2 & x_0x_1 & 3x_3^2\\
(k+1)x_0^k & 0 & 0 & (k+1)x_3^k
\end{pmatrix}.
$$
For $k = 2$, we note that the matrix below:
$$
\begin{pmatrix}
0 & -x_0x_3^2\\
x_1   & 0\\
-x_2 & x_2x_3^2 \\
0   & x_0^3
\end{pmatrix}
$$
gives trivializing syzygies such that $\calt_\sigma \simeq \mathcal{O}_{\mathbb{P}^3}(-1) \oplus \mathcal{O}_{\mathbb{P}^3}(-3)$.

For $k > 2$, we set
$$
M = \begin{pmatrix}
0      & -x_0x_2x_3^k                &-x_0x_1x_3^k\\
x_1    & -3x_0^kx_3^2 + 3x_0^2x_3^k  & 0   \\
-x_2   & x_2^2x_3^k                  & -3x_0^kx_3^2 + 3x_0^2 x_3^k + x_1x_2 x_3^k   \\
0      & x_0^{k+1}x_2                & x_0^{k+1}x_1
\end{pmatrix}, 
\gamma = \begin{pmatrix}
x_0^k x_3^2 - x_0^2 x_3^k\\    
\frac{1}{3}x_1\\
- \frac{1}{3} x_2
\end{pmatrix},
$$
so we obtain a free resolution of $\calt_\sigma$ given by
$$
0 \rightarrow \mathcal{O}_{\mathbb{P}^3}(-d-1) \xrightarrow{\gamma} \mathcal{O}_{\mathbb{P}^3}(-1) \oplus \mathcal{O}_{\mathbb{P}^3}(-d)^{\oplus 2} \xrightarrow{M} \mathcal{T}_\sigma \rightarrow 0,
$$
and in particular $\sigma$ is a nearly-free sequence.
\end{example}

\begin{prop}\label[prop]{prop:nearly-free-not-loc-free}
If $\sigma=(f,g)$ is a nearly-free normal sequence, then $\calt_\sigma$ is not locally free.
\end{prop}

\begin{proof}
Since
$$
c_3(\calt_\sigma) = 2p_a(B) - 2 + \deg(B)(4+d-2e),
$$
assuming $\calt_\sigma$ is locally free, we obtain $c_3(\calt_\sigma) = 0$. On the other hand, since $\sigma$ is nearly free, $\deg(B) = 1$ and $p_a(B) = 0$, hence
$$
e = \frac{d+2}{2},
$$
which implies that $d$ must be even and that $e = d/2 + 1$, giving $h^0(\calt_\sigma(l)) = 0$ whenever $l \leq d/2 = -\mu(\calt_\sigma)$, hence $\calt_\sigma$ is $\mu-$semistable.

If we denote by $E \doteq \calt_\sigma(d/2)$ the normalized vector bundle, we must have $c_1(E) = 0$ and
\begin{align*}
c_2(E) &= c_2(\calt_\sigma) - \frac{d^2}{2} + \frac{d^2}{4}\\
&= m_0 - m(\sigma) - \frac{d^2}{4}.
\end{align*}
On the other hand, using the Bourbaki degree formula with $\Bour(\sigma) = 1$ and $e = \frac{d+2}{2}$, we obtain
\begin{align*}
1=\Bour(\sigma) &= \frac{(d+2)^2}{4}-\frac{d(d+2)}{2}+m_0 - m(\sigma)\\
&= \frac{d^2}{4} - \frac{d^2}{2} - d+\frac{4d}{4}+1+m_0 - m(\sigma),
\end{align*}
which yields $
m_0-m(\sigma)-\frac{d^2}{4} = 0$, so that $c_2(E) = 0$. But a $\mu-$semistable reflexive sheaf with $c_1(E) = c_2(E) = 0$ must be $E \simeq \mathcal{O}_{\mathbb{P}^3}^{\oplus 2}$ (by \cite[Lemma 9.7]{Hartshorne1980}), implying $\sigma$ is free, contradicting the hypothesis.
\end{proof}

\begin{prop}\label[prop]{prop:pog-seq}
If a normal sequence $\sigma = (f,g)$ is a $3$-syzygy, then $\gpdim(\calt_\sigma) = 1$. Moreover, a sequence $\sigma$ is $3$-syzygy if and only if $B_\nu$ is a complete intersection, for $\nu \in H^0(\calt_\sigma(e)) \setminus \{0\}$, $e = \indeg(\calt_\sigma)$.
\end{prop}

\begin{proof}
First, if we assume $\sigma$ is $3$-syzygy, then there is a free resolution of the form:
$$
0 \rightarrow F_2 \rightarrow F_1 \rightarrow F_0 \oplus \mathcal{O}_{\mathbb{P}^3}(-e) \xrightarrow{\lambda} \calt_\sigma \rightarrow 0
$$
with $\rank(F_0) = 2$. Split the resolution into two short exact sequences:
$$
F_2 \hookrightarrow F_1 \twoheadrightarrow S \text{ and } S \hookrightarrow F_0 \oplus \mathcal{O}_{\mathbb{P}^3}(-e) \twoheadrightarrow \calt_\sigma,
$$
and focus on the second one. The sheaf $S$ is the kernel of a map between a locally free sheaf $F_0'$ and a torsion-free sheaf $\calt_\sigma$, thus $S$ is reflexive, from \cite[Proposition 1.1]{Hartshorne1980}. Furthermore, since $\rank(F_0) = 2$ and $\rank(\calt_\sigma) = 2$, $S$ is a reflexive sheaf of rank one, thus $S \simeq \mathcal{O}_{\mathbb{P}^3}(-k)$ for some $k \in \mathbb{Z}$, hence
$$
S \simeq \mathcal{O}_{\mathbb{P}^3}(-k) \hookrightarrow F_0' \twoheadrightarrow \calt_\sigma
$$
is a free resolution for $\calt_\sigma$, concluding $\gpdim(\calt_\sigma) = 1$.

With this and \Cref{prop:lifting-resolutions}, we conclude that $\sigma$ is $3$-syzygy if and only if $B = B_\nu$ admits a resolution of the form:
$$
0 \to \mathcal{O}_{\mathbb{P}^3}(-k) \to \mathcal{O}_{\mathbb{P}^3}(-a) \oplus \mathcal{O}_{\mathbb{P}^3}(-b) \to \mathcal{I}_B \to 0,
$$
which is equivalent to the fact that $B \subset \mathbb{P}^3$ is a complete intersection.
\end{proof}

\begin{remark}\label[remark]{rmk:chain-of-implications-geometry-of-B}
We note that there is a chain of implications:
$$
\sigma \text{ is nearly free } \Rightarrow \sigma \text{ is $3$-syzygy } \Rightarrow \gpdim(\calt_\sigma) = 1,
$$
from \Cref{prop:nearly-free-seq} and \Cref{prop:pog-seq}. The converses do not hold, as we explore in the examples: there are $3$-syzygy pencils of cubics which are not nearly free (\Cref{ex:pcubics-pog-not-nf}), and the next example is a pencil of cubics satisfying $\gpdim(\calt_\sigma) = 1$ but not $3$-syzygy.

\end{remark}

\begin{example}[$\gpdim(\calt_\sigma) = 1$, not $3$-syzygy]\label[example]{ex:pencilcubics-Bour4-m5-notpog}
Consider the following pencil of cubics:
$$
\sigma = (x_0^2x_2 + x_0x_1x_3 + x_3^3, x_2^3 + x_1x_2x_3+x_3^3)
$$
with Jacobian matrix
$$
\nabla \sigma = \begin{pmatrix}
2x_0x_2 + x_1 x_3 & x_0x_3 & x_0^2 & x_0x_1 + 3x_3^2\\
0 & x_2x_3 & 3x_2^2 + x_1 x_3 & x_1x_2+ 3x_3^2
\end{pmatrix}.
$$
The scheme $(\Xi_\sigma)_1$ has three irreducible components, given by two lines $V(x_0, x_3), 
V(x_2, x_3)$ and a plane quadric curve $V(x_0-x_2, 2x_2^2 + x_1x_3)$. Using the associativity formula, one may compute $m(\sigma) = 5$.

The minimal free resolution of $\calt_\sigma$ given using Macaulay2 is
$$
0 \rightarrow \mathcal{O}_{\mathbb{P}^3}(-4) \oplus \mathcal{O}_{\mathbb{P}^3}(-5) \xrightarrow{M} \mathcal{O}_{\mathbb{P}^3}(-3)^{\oplus 3} \oplus \mathcal{O}_{\mathbb{P}^3}(-4) \rightarrow \calt_\sigma \rightarrow 0,
$$
so that $e = 3$ and $\Bour(\sigma) = 4$. This means $\gpdim(\calt_\sigma) = 1$, but we need $4$ syzygies to generate $\calt_\sigma$ or, equivalently, the degree four Bourbaki scheme is not a complete intersection.
\end{example}

\begin{example}[pencil of cubics with $\gpdim(\calt_\sigma) = 2$]\label[example]{ex:pencil-of-cubics-bour2}
We consider the sequence of cubics $d_f = d_g = 2$ given by:
$$
\sigma = (x_0x_1^2+x_2^3+x_2^2x_3, x_2x_3(x_2 -x_1)),
$$
considered in \cite[Theorem 8.1]{faenzi:hal-03271244}. From their proof, we know that $\calt_\sigma(2)$ is a null correlation bundle. Therefore, we obtain $\Bour(\sigma) = 2$, and a free resolution for $\calt_\sigma$ is given by:
$$
0 \rightarrow \mathcal{O}_{\mathbb{P}^3}(-5) \to \mathcal{O}_{\mathbb{P}^3}(-4)^{\oplus 4} \to \mathcal{O}_{\mathbb{P}^3}(-3)^{\oplus 5} \rightarrow \calt_\sigma \rightarrow 0.
$$
So $e = 3$ and $m(\sigma) = 7$. The irreducible components of $(\Xi_\sigma)_1$ are three lines $V(x_2, x_3), V(x_1 - x_2, x_3)$ and $V(x_1, x_2)$.
\end{example}

For normal sequences $\sigma = (f,g)$, the triples $(\indeg(\calt_\sigma, \Bour(\sigma), m(\sigma))$ are not enough ensure the homological behavior of the sheaf $\calt_\sigma$. The resolutions for the associated Bourbaki schemes also fix the arithmetic genus $p_a(B)$, which describes the third Chern class $c_3(\calt_\sigma)$ of the sheaf. Even if these are fixed, there are curves with fixed degree and genus with different minimal free resolutions. The next example has the same triple $(\indeg(\calt_\sigma, \Bour(\sigma), m(\sigma))$ as in the $3$-syzygy \Cref{ex:pcubics-pog-not-nf} seen before, but this one has $\gpdim(\calt_\sigma) = 2$.

\begin{example}\label[example]{ex:pencilcubics-Bour4-m5-pog-c3-8}
$$
\sigma = (x_0^3 + x_0x_1x_3 + x_3^3, x_3^3 + x_1x_3^2 + x_0x_1x_3 + x_0^2x_2),
$$
with the associated Jacobian matrix given by:
$$
\nabla \sigma = \begin{pmatrix}
3x_0^2+ x_1x_3 & x_0x_3 & 0 & x_0x_1+3x_3^2\\
2x_0x_2+x_1x_3 & x_0x_3+x_3^2& x_0^2 & x_0x_1+2x_1x_3+3x_3^2
\end{pmatrix}.
$$
Here, the unique irreducible component of $(\Xi_\sigma)_1$ is the line $V(x_0, x_3)$. Using Macaulay2, we obtain a free resolution for $\calt_\sigma$ of the form:
$$
0 \rightarrow \mathcal{O}_{\mathbb{P}^3}(-6)^{\oplus 2} \rightarrow \mathcal{O}_{\mathbb{P}^3}(-5)^{\oplus 7} \rightarrow \mathcal{O}_{\mathbb{P}^3}(-4)^{\oplus 6} \oplus \mathcal{O}_{\mathbb{P}^3}(-3) \rightarrow \calt_\sigma \rightarrow 0,
$$
so that $e = 3$, $\Bour(\sigma) = 4$ and $m(\sigma) = 5$, with $c_3(\calt_\sigma) = 8$ and $\gpdim(\calt_\sigma) = 2$.
A resolution for $\mathcal{I}_B$ will be of the form:
$$
0 \rightarrow \mathcal{O}_{\mathbb{P}^3}(-5)^{\oplus 2} \rightarrow \mathcal{O}_{\mathbb{P}^3}(-4)^{\oplus 7} \rightarrow \mathcal{O}_{\mathbb{P}^3}(-3)^{\oplus 6} \rightarrow \mathcal{I}_B \rightarrow 0.
$$
\end{example}

In \Cref{sec:special-families} we characterize all nearly-free pencils of cubics $(d_f = d_g = 2)$ and all nearly-free sequences with $d_f = 1, d_g = 2$.

\section{Extreme cases of low initial degree}\label[section]{sec:lower-initial-indeg}

As explored in \cite[Section 9]{faenzi:hal-03271244}, a sequence $\sigma = (f,g)$ defines a codimension-one foliation with a corresponding short exact sequence:
\begin{align*}
\mathcal{F}_\sigma : 0 \rightarrow \calt_\sigma(1) \rightarrow \mathbb{T}\mathbb{P}^3 \rightarrow \mathcal{I}_{\Gamma_\sigma}(d+2) \rightarrow 0,
\end{align*}
where $\Gamma_\sigma \subset \mathbb{P}^3$ is called the \emph{singular scheme} of $\mathcal{F}_\sigma$. Then, assuming $e = \indeg(\calt_\sigma)$, there is a nonzero section of $\nu \in \text{Hom}(\mathcal{O}_{\mathbb{P}^3}(1-e), \calt_\sigma(1))$, inducing the commutative diagram with exact rows below
\[\begin{tikzcd}
	& {\mathcal{O}_{\mathbb{P}^3}}(1-e) & {\mathcal{O}_{\mathbb{P}^3}}(1-e) \\
	0 & {\calt_\sigma(1)} & {\mathbb{T}\mathbb{P}^3} & {\mathcal{I}_{\Gamma_\sigma}(d+2)} & 0 \\
	0 & {\mathcal{I}_B(e-d+1)} & G & {\mathcal{I}_{\Gamma_\sigma}(d+2)} & 0
	\arrow[Rightarrow, no head, from=1-2, to=1-3]
	\arrow[hook, from=1-2, to=2-2]
	\arrow[hook, from=1-3, to=2-3]
	\arrow[from=2-1, to=2-2]
	\arrow[from=2-2, to=2-3]
	\arrow[two heads, from=2-2, to=3-2]
	\arrow[from=2-3, to=2-4]
	\arrow[two heads, from=2-3, to=3-3]
	\arrow[from=2-4, to=2-5]
	\arrow[Rightarrow, no head, from=2-4, to=3-4]
	\arrow[from=3-1, to=3-2]
	\arrow[from=3-2, to=3-3]
	\arrow[from=3-3, to=3-4]
	\arrow[from=3-4, to=3-5]
\end{tikzcd}\]
The middle column of the previous diagram
$$
0 \rightarrow \mathcal{O}_{\mathbb{P}^3}(1-e) \rightarrow \mathbb{T}\mathbb{P}^3 \rightarrow G \rightarrow 0
$$
defines a foliation by curves of $\mathbb{P}^3$ of degree $e$, where $G$ is a rank two torsion-free sheaf and $G^\vee$ is called the \emph{conormal sheaf} of the foliation. Dualizing this short exact sequence, we obtain
$$
0 \rightarrow G^\vee \rightarrow \Omega^1_{\mathbb{P}^3} \rightarrow \mathcal{I}_W(e-1) \rightarrow 0,
$$
defining a subscheme $W \subset \mathbb{P}^3$, called the \emph{singular scheme} of the associated foliation by curves. It has codimension at least two and it is also described by $\inext^1(G, \mathcal{O}_{\mathbb{P}^3}) \simeq \mathcal{O}_W(e-1)$.

We derive numerical restrictions for $e \in \{1, 2\}$ using the classification of foliations by curves in $\mathbb{P}^3$ of degrees one and two. A similar strategy was employed to study distributions on $\mathbb{P}^3$ with unstable tangent sheaf in \cite{barbassa2026distributionsunstabletangentsheaf}, and general considerations about the minimal degree of sub-foliations have been considered in \cite{Mendson_Pereira_2024}.

One may localize the diagram above at a minimal prime of the scheme $B$, arguing as in \cite[Proposition 3.4]{correa2026holomorphicfoliationsdegreearbitrary}. Over this prime, both sheaves $\calt_\sigma$ and $\mathbb{T}\mathbb{P}^3$ split as sums of line bundles, and we obtain a diagram
\[\begin{tikzcd}
	{\mathcal{O}_{\mathbb{P}^3, p}(1-e)} & {\mathcal{O}_{\mathbb{P}^3, p}(1-e)} \\
	{\mathcal{O}_{\mathbb{P}^3, p}(-a) \oplus \mathcal{O}_{\mathbb{P}^3, p}(-b)} & {\bigoplus_{i=1}^3 \mathcal{O}_{\mathbb{P}^3, p}(l_i)} \\
	{\mathcal{I}_B(e-d+1)_p} & {G_p}
	\arrow[equals, from=1-1, to=1-2]
	\arrow["\nu", from=1-1, to=2-1]
	\arrow["{\nu'}"', from=1-2, to=2-2]
	\arrow["L", from=2-1, to=2-2]
	\arrow[two heads, from=2-1, to=3-1]
	\arrow[two heads, from=2-2, to=3-2]
\end{tikzcd}\]
describing $W$ and $B$ via matrix presentations, as the vanishing set of $\nu'$ and $\nu$, respectively. From the composition of matrices $L \cdot \nu = \nu'$ we conclude minors of $\nu'$ are factors of the minors of $\nu$, and thus $B \subset W$ at every associated prime, hence $B \subset W$ schematically. Now, we turn to the classification of schemes $W$ for foliations by curves on $\mathbb{P}^3$ of degrees $e = 1, 2$:
\begin{Lemma}\label[Lemma]{lem:classification-foliation-by-curves}
Let
$$
\mathcal{F} : 0 \to \mathcal{O}_{\mathbb{P}^3}(1-e) \to \mathbb{T}\mathbb{P}^3 \to G \to 0
$$
be a foliation by curves of degree $e \geq 0$ in $\mathbb{P}^3$, with singular scheme $W \subset \mathbb{P}^3$. Then:
\begin{itemize}
    \item[(a)](\cite[Theorem $4$]{galeano2022codimension}) If $e = 1$, then $W$ is either a $0$-dimensional scheme of length $4$, a union of a line with a zero-dimensional scheme of length two or a pair of skew lines, in which case $G^\vee \simeq \mathcal{O}_{\mathbb{P}^3}(-2)^{\oplus 2}$. In particular, $\deg(W) \leq 2$.
    \item[(b)](in preparation, V. Cordeiro) If $e = 2$, then $\deg(W) \leq 5$.
\end{itemize}    
\end{Lemma}

\begin{proof}
For completeness, we include an argument for $(b)$, from V. Cordeiro: using the short exact sequence
$$
0 \to G^\vee \to \Omega^1_{\mathbb{P}^3} \to \mathcal{I}_W(1) \to 0,
$$
one obtains $c_2(G^\vee) = 11-\deg(W)\leq 11$ (see, for example, \cite[4.1, 4.2]{correa2023classification}).If $G^\vee$ is stable, Bogomolov's inequality says $c_2(G^\vee) \geq 7$, hence $\deg(W) \leq 4$. Now, assuming that $G^\vee$ is not stable, $h^0(G^\vee(2)) \neq 0$ and one may choose a nonzero section to form a sequence
$$
0 \to \mathcal{O}_{\mathbb{P}^3}(-2) \to G^\vee \to \mathcal{I}_Y(-3) \to 0
$$
where $Y$ is the vanishing locus of this section. From the sequence, we get $0 \leq \deg(Y) = c_2(\mathcal{I}_Y) = c_2(G^\vee) - 6$, so $c_2(G^\vee) \geq 6$, and $\deg(W) \leq 5$.    
\end{proof}

\phantomsection\label{thm:second}
\begin{maintheorem}\label[maintheorem]{thmB}
\textit{Let $\sigma = (f,g)$ be a normal sequence of polynomials of degrees $d_f + 1, d_g + 1$. Then:
\begin{itemize}
    \item[(a)] If $\indeg(\calt_\sigma) = 1$, then $\Bour(\sigma) \in \{0,  1, 2\}$;
    \item[(b)] If $\indeg(\calt_\sigma) = 2$, then $\Bour(\sigma) \leq 5$.
\end{itemize}}
\end{maintheorem} 

\begin{proof}
The result follows from the schematic inclusion $B \subset W$ and from the classification results for the singular scheme $W$, since $\deg(B) \leq \deg(W)$.
\end{proof}

To be able to present the next applications, we will need the following result on the structure of degree $2$ space curves in $\mathbb{P}^3$:

\begin{theorem}\cite[1.4-1.6]{Nollet1997}\label[theorem]{thm:mult-2-lines}
Let $B \subset \mathbb{P}^3$ be a curve of degree $2$ and genus $p_a(B) = -1-a$, for $a \in \mathbb{Z}$. Then:
\begin{itemize}
    \item[(a)] $a \geq -1$, and $a = -1$ if and only if $B$ is planar; 
    \item[(b)] For $a \geq 1$, $B$ must be a multiplicity two structure at a line $L \subset \mathbb{P}^3$, and these satisfy a short exact sequence of the form
    $$
    0 \rightarrow \mathcal{O}_L(a) \rightarrow \mathcal{O}_B \rightarrow \mathcal{O}_L \rightarrow 0;
    $$
    \item[(c)] For $a \geq 1$, if $B$ is a multiplicity two structure on a line, then $\omega_B \simeq \mathcal{O}_{B}(-a-2)$. If $a=0$ and $B$ is a union of two skew lines, then $\omega_B \simeq \mathcal{O}_{B}(-2)$.
\end{itemize}
\end{theorem}

For the case $e = 1$ and $\Bour(\sigma) = 2$, one gets $B = W$ and we can predict the Chern classes for the logarithmic sheaf, by the following proposition. We note that we have not been able to construct examples of sequences of this kind.

\begin{prop}\label[prop]{prop:e-1-Bour-2-invariants}
Let $\sigma = (f,g)$ be a normal sequence with $e = 1$ and $\Bour(\sigma) = 2$. Then, the Chern classes of $\calt_\sigma$ are $(-d, d+1, 2d)$.
\end{prop}

\begin{proof}
From the formulas $c_2(\calt_\sigma) = m_0 - m(\sigma)$ and $\Bour(\sigma) = 2$, we get $c_2(\calt_\sigma) = d+1$. Since $B = W$ is a pair of skew lines in this case, $p_a(B) = -1$ and from
$$
c_3(\calt_\sigma) = 2p_a(B) - 2 + \deg(B) (4 + d - 2e)
$$
it follows that $c_3(\calt_\sigma) = 2d$.
\end{proof}

The last proposition of the section also uses the foliation structure, and it follows from \cite[Proposition B]{correa2026holomorphicfoliationsdegreearbitrary}.

\begin{prop}
Let $\sigma = (f,g)$ be a normal sequence with $\indeg(\calt_\sigma) = 1$. If $\sigma$ is not free, then $\calt_\sigma$ is not locally free.
\end{prop}

\section{Pencils of cubics and degree 6 curves inside quadric surfaces}\label[section]{sec:special-families}

In this final section, we show some classification results for pencils of cubics and sequences $\sigma = (f,g)$ with $d_f = 1$, $d_g = 2$, which correspond to degree $6$ curves inside quadric surfaces. The results are derived from the previous sections, \Cref{subsec:bourbaki-degree-first} and \Cref{sec:lower-initial-indeg}, and also from general results for reflexive sheaves of rank two on $\mathbb{P}^3$, found in the classical works \cite{Hartshorne1980}, \cite{sols1981stable}, \cite{Hartshorne1982}, \cite{chang1984stable} and \cite{hartshorne1988stable}. We recall two major results:

\begin{theorem}[\cite{hartshorne1988stable}, Theorem 1.1]\label[theorem]{thm:hartshorne1988}
Let $\mathcal{E}$ be a rank two reflexive sheaf on $\mathbb{P}^3$. Assume $c_1 \geq -3$ and $h^0(\mathcal{E}) = 0$. Define the integers:
\begin{align*}
A &\doteq \left\lceil \frac{1}{3}(c_1^2+2c_1 + 3) \right\rceil, \text{resp. ditto +1 if }c_1 = 1, 3;\\
B &\doteq \left\lceil \frac{1}{3} (c_1^2 + 3c_1 + 8)  \right\rceil, \text{resp. ditto +1 if }c_1 = 2, 4; \text{ditto -1 if }c_1 = -3.
\end{align*}
Then $c_2 \geq A$. Furthermore:
\begin{itemize}
    \item[(a)] If $A \leq c_2 \leq B$, then
    $$
    c_3 \leq (c_1 + 4)c_2 - 2\binom{c_1 + 3}{3}-2.
    $$
    \item[(b)] If $c_2 > B$, then
    $$
    c_3 \leq c_2^2 - c_2(2B - c_1 - 5) + B^2 - B - 2\binom{c_1 + 3}{3} - 2.
    $$
\end{itemize}
\end{theorem}

With this theorem, we obtain an upper bound for $m(\sigma)$ lower than $m_0$ when $\sigma$ is incompressible, so that $h^0(\calt_\sigma) = 0$, in the case $d_f = 1, d_g = 2$. In the case of pencils of cubics, however, we apply this to $\mathcal{E} = \calt_\sigma(1)$, so we must assume $\indeg(\calt_\sigma) > 1$ and treat the case $\indeg(\calt_\sigma) = 1$ separately. 

Another important result is the following, which enables us to obtain a lower bound for the existence of a syzygy in some cases:

\begin{theorem}[\cite{Hartshorne1982}, Theorem 0.1]\label[theorem]{thm:hartshorne1982}
Let $\mathcal{E}$ be a rank two reflexive sheaf on $\mathbb{P}^3$ with $c_1 = 0$ of $c_1 = -1$ and with $c_2 \geq 0$. Let $t \in \mathbb{Z}$ such that either
\begin{itemize}
    \item[(a)] $c_1 = 0$ and $t > \sqrt{3c_2 +1}-2$, or
    \item[(b)] $c_1 = -1$ and $t > \sqrt{3c_2 + \frac{1}{4}} - \frac{3}{2}$.
\end{itemize}
Then $H^0(\mathcal{E}(t)) \neq 0$.
\end{theorem}

\subsection{Pencils of cubics}

Now, assume $d_f = d_g = 2$. The Bourbaki degree of a sequence $\sigma$ in terms of $e = \indeg(\calt_\sigma)$ is given by the formula
$$
\Bour(\sigma) = e(e-4) + 12 - m(\sigma).
$$
When $e = 1$, $\Bour(\sigma) \geq 0$ gives $m(\sigma) \leq 9$. Applying \Cref{thm:hartshorne1988} to the sheaf $\mathcal{E} = \calt_\sigma(1)$, assuming $e > 1$, we get $c_1(\mathcal{E})= -2$ and the quantities of the Theorem become $A = 1$ and $B = 2$. Since $
c_2(\mathcal{E}) =9 - m(\sigma)$ direct application of the result shows the next proposition.

\begin{prop}\label[prop]{prop:bounds-c3-pencilcubics}
If $\sigma = (f,g)$ is a normal pencil of cubics and $e = \indeg(\calt_\sigma) > 1$, then $m(\sigma) \leq 8$. Furthermore:
\begin{itemize}
    \item[(a)] if $7 \leq m(\sigma) \leq 8$, then $c_3 \leq 16 - 2m(\sigma)$. In particular, when $m(\sigma) = 8$, $\sigma$ is locally free.
    \item[(b)] if $0 \leq m(\sigma) < 7$, then $
    c_3 \leq m(\sigma)^2 - 17m(\sigma) + 72$.
\end{itemize}
\end{prop}

Next, we apply \Cref{thm:hartshorne1982} to the sheaf $\mathcal{E} = \calt_\sigma(2)$, a sheaf with $c_1 = 0$ and $c_2(\mathcal{E}) = 8 - m(\sigma) \geq 0$, assuming $e > 1$. Then, for
$
t > \sqrt{3c_2 + 1}-2 = \sqrt{24 - 3m(\sigma) + 1}-2
$ we have $h^0(\mathcal{E}(t)) = h^0(\calt_\sigma(t+2)) \neq 0$. Direct application of this bound for high values of $m(\sigma) = 6,7,8$ gives the following result.

\begin{prop}\label[prop]{prop:indeg-bounds-pencilcubics}
Let $\sigma = (f,g)$ be a normal pencil of cubics. Then $e = \indeg(\calt_\sigma) \leq 4$. If $m(\sigma) = 6, 7$, $e \leq 3$, and if $m(\sigma) = 8$, $e \leq 2$.
\end{prop}

The following result gives a picture of the generic case of a pencil of cubics.

\begin{prop}\label[prop]{generic-pencil-of-cubics}
Let $\sigma = (f,g)$ be a general pencil of cubics in $\mathbb{P}^3$. Then $m(\sigma) = 0$, the number of singular members is $32$, and all singular members have one singular point, in particular $m(\sigma) = 0$. Moreover $c_3(\calt_\sigma) = 32$ (see \Cref{ex:pencils-cubics-genericpencil}), and the minimal free resolution for the logarithmic sheaf $\calt_\sigma$ is of the form
$$
0 \to \mathcal{O}_{\mathbb{P}^3}(-6)^{\oplus 2} \to \mathcal{O}_{\mathbb{P}^3}(-4)^{\oplus 4} \to \calt_\sigma \to 0,
$$
given by the Buchsbaum--Rim complex.
\end{prop}

\begin{proof}
This follows from intersection theory for the bundle of principal parts (see \cite{eisenbud20163264}, Proposition $7.1$ and Proposition $7.4$).
\end{proof}

Rewriting \Cref{thmA}, $(d)$ for pencils of cubics, we obtain:

\begin{prop}\label[prop]{prop:pencilcubics-compressible}
Let $\sigma$ be a normal pencil of cubics. Then $\sigma$ is compressible if and only if $m(\sigma) = 12$ (see \Cref{ex:compressible-pencilcubics}), and in this case $\calt_\sigma \simeq \mathcal{O}_{\mathbb{P}^3}\oplus \mathcal{O}_{\mathbb{P}^3}(-4)$.
\end{prop}

\begin{example}[Free, compressible pencil of cubics]\label[example]{ex:compressible-pencilcubics}
Consider the sequence $\sigma = (x_0^3 + x_1^3 + x_0x_1x_3, x_0x_1x_3)$. This sequence is independent of the variable $x_2$, with Jacobian matrix
$$
\nabla \sigma = \begin{pmatrix}
3x_0^2 + x_1x_3       & 3x_1^2+x_0x_3 &0 & x_0x_1\\
x_1x_3 & x_0x_3  &0 & x_0x_1
\end{pmatrix}
$$
so that there are two linearly independent syzygies, one of degree zero and one of degree four, in the following matrix:
$$
\nu = \begin{pmatrix}
0 & -x_0x_1^3\\
0 & x_0^3x_1\\
1 & 0 \\
0 & -x_0^3x_3 + x_1^3 x_3
\end{pmatrix}
$$
Here, $m(\sigma) = 12$, with the irreducible components of $(\Xi_\sigma)_1$ being three lines $V(x_0, x_3), V(x_1, x_3)$ and $V(x_0,x_1)$, where the last one has multiplicity $10$ and the other two are simple.
\end{example}

Using \Cref{thmB}, we obtain the following bounds for $\mu$-semistability of $\calt_\sigma$ in terms of $m(\sigma)$:

\begin{prop}\label[prop]{prop:class-unstable-pencilcubics}
Let $\sigma$ be a normal pencil of cubics. Then:
\begin{itemize}
    \item[(a)] If $m(\sigma) \leq 6$, then $\calt_\sigma$ is $\mu$-semistable;
    \item[(b)] If $m(\sigma) \leq 2$, then $\calt_\sigma$ is $\mu$-stable.
\end{itemize}
\end{prop}

\begin{proof}
Since $\mu(\calt_\sigma) = -2$, we show for $m(\sigma) \leq 6$ that $e = \indeg(\calt_\sigma) \geq 2$. Since $m(\sigma) \leq 6$, we conclude that $\sigma$ is neither compressible nor free, from the previous results \Cref{prop:pencilcubics-compressible}, \Cref{prop:class-pencilcubics-m9} and \Cref{prop:class-pencilcubics-m8}. Let us suppose that $e = \indeg(\calt_\sigma) = 1$. Since $\sigma$ is not free, we can apply the result \Cref{thmB}, $(a)$, and conclude $\Bour(\sigma) \leq 2$. On the other hand, we have
\begin{align*}
\Bour(\sigma) &= 1 - d + d_f^2 + d_g^2 + d_f d_g - m(\sigma)\\
&= 1 - 4 + 4 + 4 + 4 - m(\sigma)\\
&= 9 - m(\sigma) > 2,
\end{align*}
since $m(\sigma) \leq 6$. This contradicts the bound $\Bour(\sigma) \leq 2$ established earlier.

Moreover, for $(b)$, if we assume $m(\sigma) \leq 2$ and $e = \indeg(\calt_\sigma) = 2$, then
\begin{align*}
\Bour(\sigma) &= 2 - 2d + d_f^2 + d_g^2 + d_f d_g - m(\sigma)\\
&= 4 - 8 + 12 - m(\sigma)\\
&= 8 - m(\sigma) > 5
\end{align*}
if $m(\sigma) \leq 2$, we got a contradiction with the bound $\Bour(\sigma) \leq 5$.
\end{proof}

\begin{prop}\label[prop]{prop:class-pencilcubics-m9}
Let $\sigma = (f,g)$ be an incompressible normal pencil of cubics. Then $m(\sigma) \leq 9$, and $m(\sigma) = 9$ if and only if $\calt_\sigma \simeq \mathcal{O}_{\mathbb{P}^3}(-1) \oplus \mathcal{O}_{\mathbb{P}^3}(-3)$ (see \Cref{ex:free-incompressible-m9}).
\end{prop}

\begin{proof}
The bound $m(\sigma) \leq 8$ is obtained for $e = \indeg(\calt_\sigma) > 1$ in \Cref{prop:bounds-c3-pencilcubics} for pencils of cubics, and $m(\sigma) \leq 9$ holds for $e \geq 1$, thus $m(\sigma) = 9$ only if $e = 1$. From the formula of the Bourbaki degree, we obtain $\Bour(\sigma) = 0$ in this case, and thus $\calt_\sigma \simeq \mathcal{O}_{\mathbb{P}^3}(-1) \oplus \mathcal{O}_{\mathbb{P}^3}(-3)$. On the other hand, if $\calt_\sigma \simeq \mathcal{O}_{\mathbb{P}^3}(-1) \oplus \mathcal{O}_{\mathbb{P}^3}(-3)$, then $e = 1$, and the equation $\Bour(\sigma) = 0$ implies $m(\sigma)=9$.
\end{proof}

\begin{example}[Free, unstable and incompressible pencil of cubics]\label[example]{ex:free-incompressible-m9}
Consider the sequence
$\sigma = (x_1(x_2^2 - x_1^2), x_3x_2(x_0 - x_1))$. Then the matrix $\nabla \sigma$ is given by:
$$
\nabla \sigma = \begin{pmatrix}
0       & -3x_1^2+x_2^2 &2x_1x_2 & 0\\
x_2x_3 & -x_2x_3  &x_3(x_0 - x_1) & x_2(x_0 - x_1)
\end{pmatrix}
$$
and it admits two linearly independent syzygies, one of degree one and one of degree $3$:
$$
\nu \doteq \begin{pmatrix}
x_0 - x_1         & 2x_1x_2^2 \\
0         & 2x_1x_2^2 \\
0         & 3x_1^2x_2 - x_2^3 \\
-x_3      & -3x_1^2x_3 + x_2^2x_3 
\end{pmatrix}
$$
Thus, we conclude $\calt_\sigma \simeq \mathcal{O}_{\mathbb{P}^3}(-1) \oplus \mathcal{O}_{\mathbb{P}^3}(-3)$, and in particular $e = \indeg(\calt_\sigma) = 1$. In this case, $m(\sigma) = 9$.
\end{example}

\begin{prop}\label[prop]{prop:class-pencilcubics-m8}
Let $\sigma = (f,g)$ be a normal pencil of cubics with $m(\sigma) = 8$. Then either:
\begin{itemize}
    \item[(a)] $\sigma$ is free, with $\calt_\sigma \simeq \mathcal{O}_{\mathbb{P}^3}(-2) \oplus \mathcal{O}_{\mathbb{P}^3}(-2)$ (see \Cref{ex:free-incompressible-m8});
    \item[(b)] $\sigma$ is nearly-free, with resolution
    $$
    0 \to \mathcal{O}_{\mathbb{P}^3}(-5) \to \mathcal{O}_{\mathbb{P}^3}(-1) \oplus \mathcal{O}_{\mathbb{P}^3}^{\oplus 4} \to \calt_\sigma \to 0.
    $$
    For an example, see \Cref{ex:m8-nearlyfree}
\end{itemize}
\end{prop}

\begin{proof}
If $m(\sigma) = 8$ and $e > 1$, by \Cref{prop:indeg-bounds-pencilcubics}, then $e = \indeg(\calt_\sigma) \leq 2$, and from the Bourbaki formula it follows $\Bour(\sigma) = 0$, hence the pencil is free of the splitting claimed. If $e = 1$, then $\Bour(\sigma) = 1$, and the nearly-free type resolution follows from \Cref{prop:nearly-free-seq}.
\end{proof}

\begin{example}[Free, incompressible and $\mu$-semistable pencil of cubics ($m(\sigma) = 8$)]\label[example]{ex:free-incompressible-m8}
Consider the sequence
$\sigma = (x_0^2x_1 + x_3^3, x_0^3+x_0x_2x_3 + x_3^3)$. The Jacobian matrix $\nabla \sigma$ is given by:
$$
\nabla \sigma = \begin{pmatrix}
2x_0x_1       & x_0^2 &   0 & 3x_3^2\\
3x_0^2+x_2x_3 & 0  & x_0x_3&x_0x_2+3x_2^2
\end{pmatrix}
$$
and it admits two linearly independent syzygies of degree $2$:
$$
\nu \doteq \begin{pmatrix}
-x_0x_3         & -x_0x_2 \\
2x_1x_3         & 2x_1x_2-9x_3^2 \\
3x_0^2 + x_2x_3 & x_2^2 - 9x_0x_3 \\
0               & 3x_0^2
\end{pmatrix}
$$
Thus, we conclude $\calt_\sigma \simeq \mathcal{O}_{\mathbb{P}^3}(-2) \oplus \mathcal{O}_{\mathbb{P}^3}(-2)$.
\end{example}

\begin{example}[Nearly-free pencil of cubics with $m(\sigma) = 8$]\label[example]{ex:m8-nearlyfree}
Consider the sequence
$$
\sigma = (x_0x_1x_2+x_3^3, x_0x_1x_3+x_2^3).
$$
Then $\sigma$ is a pencil of cubics with Jacobian matrix
$$
\nabla\sigma=
\begin{pmatrix}
x_1x_2 & x_0x_2 & x_0x_1 & 3x_3^2\\
x_1x_3 & x_0x_3 & 3x_2^2 & x_0x_1
\end{pmatrix}
$$
so we can show it is normal and it has a degree-one syzygy $\nu = (x_0, -x_1, 0, 0)$. A direct computation using the associativity formula as previously yields $m(\sigma) = 8$.
\end{example}

\begin{example}[Nearly-free pencil of cubics with $m(\sigma) = 7$]\label[example]{ex:nearly-free-cubics}
We consider the following sequence of cubics:
$$
\sigma = (x_0^2(x_1-x_2) + x_2^2(x_1 - x_0 + x_3), -x_1x_2x_3 + x_2^2x_3)
$$
with corresponding Jacobian matrix given by:
$$
\nabla \sigma = 
\begin{pmatrix}
2x_0(x_1-x_2)-x_2^2 & x_0^2 +x_2^2 & -x_0^2+2x_2(x_1-x_0+x_3) & x_2^2\\
0 & -x_2x_3 & x_3(2x_2-x_1) & x_2(x_2-x_1)
\end{pmatrix}.
$$
There are four irreducible components in $(\Xi_\sigma)_1$, all lines, corresponding to the four prime ideals below:
\begin{align*}
\fp_1 &= (x_2, x_3)\\
\fp_2 &= (x_1 - x_2, x_3)\\
\fp_3 &= (x_1, x_2)\\
\fp_4 &= (x_0, x_2)
\end{align*}
For the prime $\fp_1$, we may consider the element $u = (2x_0(x_1-x_2) - x_2^2)^{-1} \in R_{\fp_1}$ and an elementary operation to obtain the matrix
$$
(\nabla \sigma)_{\fp_1} \sim \begin{pmatrix}
1 & u(x_0^2 + x_2^2) & u(-x_0^2 + 2x_2(x_1 - x_0+x_3)) & u x_2^2\\
0 & -x_2x_3 & x_3(2x_2-x_1) & x_2(x_2 - x_1)
\end{pmatrix},
$$
which sends the first basis vector to $(1,0) \in R_{\fp_1}^2$ and therefore we may compute the cokernel by considering the second line of the remaining matrix, namely
\begin{align*}
(\mathcal{Q}_{\sigma})_{\fp_1} &\simeq \frac{R_{\fp_1}}{(-x_2x_3, x_3(2x_2-x_1), x_2(x_2-x_1))}\\
&\simeq \frac{R_{\fp_1}}{(-x_2x_3, x_3, x_2)} \simeq \frac{R_{\fp_1}}{\fp_1R_{\fp_1}} \simeq \kappa,
\end{align*}
so that $\length(\mathcal{Q}_\sigma)_{\fp_1} = 1$.

For the prime $\fp_2 = (x_1 - x_2, x_3)$, the same element $u \in R_{\fp_2}$ is invertible, so that we compute the cokernel analogously to obtain:
\begin{align*}
(\mathcal{Q}_{\sigma})_{\fp_2} &\simeq \frac{R_{\fp_2}}{(-x_2x_3, x_3(2x_2-x_1), x_2(x_2-x_1))}\\
&\simeq \frac{R_{\fp_2}}{(x_3, x_1 - x_2)} \simeq \frac{R_{\fp_2}}{\fp_2R_{\fp_2}} \simeq \kappa,
\end{align*}
so that $\length(\mathcal{Q}_\sigma)_{\fp_2} = 1$.

For the prime $\fp_3 = (x_1, x_2)$, we take $u = (x_0^2 + x_2^2)^{-1} \in R_{\fp_3}^\times$ to perform elementary operations and obtain
$$
(\nabla \sigma)_{\fp_3} \sim \begin{pmatrix}
u(2x_0(x_1-x_2)-x_2^2) & 1 & u(-x_0^2 + 2x_2(x_1-x_0+x_3)) & ux_2^2\\
0 & 0 & x_3(2x_2-x_1)-ux_2x_3(-x_0^2+2x_2(x_1-x_0+x_3)) & x_2(x_2-x_1)-ux_2^3x_3
\end{pmatrix},
$$
which sends $e_2 \in R_{\fp_3}^4$ to $(1,0) \in R_{\fp_3}^2$. Thus, we may compute the cokernel by considering the remaining columns of the second row:
$$
(\mathcal{Q}_\sigma)_{\fp_3} \simeq \frac{R_{\fp_3}}{(f_1, f_2)}
$$
where
\begin{align*}
f_1 &= x_3(2x_2-x_1)+ux_2x_3(-x_0^2+2x_2(x_1-x_0+x_3))\\
f_2 &= x_2(x_2-x_1)+ux_2^3x_3.
\end{align*}
Since $f_1 \in I$, we obtain
$$
2x_2(1+ux_2(x_1-x_0+x_3)-ux_0^2) \equiv x_1 \mod I.
$$
Substituting this into $f_2 \in I$, we get:
$$
0 \equiv x_2^2\left(1- (2(1+ux_2(x_1-x_0+x_3))-ux_0^2)+ux_2x_3 \right) \mod I,
$$
and the element multiplying $x_2^2$ is an unit in $R_{\fp_3}$, we conclude that $x_2^2 \in I$. Hence $x_1^2 \in I$ by $f_1 \in I$, and we may rewrite $f_1$ and $f_2$ mod $I$ as:
\begin{align*}
f_1 &= x_2(2-ux_0^2)-x_1 \mod I \\
f_2 &= -x_1 x_2 \mod I,
\end{align*}
and thus the ideal $I = (x_1^2, x_2^2, x_1 x_2, x_2v - x_1)$ where $v \in R_{\fp_3}$ is a unit. Hence, one may write
$$
(\mathcal{Q}_\sigma)_{\fp_3} \simeq \kappa\langle 1, x_1, x_2 \rangle,
$$
so that $\length(\mathcal{Q}_\sigma)_{\fp_3} = 3$.

For the prime $\fp_4 = (x_0, x_2)$, we take $u = (x_3(2x_2 - x_1))^{-1} \in R_{\fp_4}$ and perform elementary matrix operations to obtain:
$$
(\nabla \sigma)_{\fp_4} \sim \begin{pmatrix}
2x_0(x_1-x_2)-x_2^2 & (x_0^2 + x_2^2)+uQx_2x_3 & 0 & x_2^2 - uQx_2(x_2-x_1)\\
0 & -ux_2x_3 & 1 & ux_2(x_2-x_1)
\end{pmatrix},
$$
sending $e_3 \in R_{\fp_4}^4$ to $(0,1) \in R_{\fp_4}^2$. Hence, we compute the cokernel as:
$$
(\mathcal{Q}_{\sigma})_{\fp_4} \simeq \frac{R_{\fp_4}}{(f_1, f_2, f_3)} \doteq \frac{R_{\fp_4}}{I}
$$
with
\begin{align*}
f_1 &= 2x_0x_1 - x_2(2x_0+x_2) = 2x_0(x_1 - x_2) - x_2^2\\
f_2 &= (x_0^2 + x_2^2) + uQx_2 x_3\\
f_3 &= x_2^2 - uQx_2(x_2 - x_1)\\
u   &= (x_3(2x_2-x_1))^{-1}\\
Q   &= -x_0^2 + 2x_2(x_1 - x_0+x_3).
\end{align*}
Thus,
\begin{align*}
f_2 &= x_0^2 + x_2^2 - ux_3 x_0^2x_2 + 2u(x_1 - x_0 + x_3)x_3x_2^2\\
&= x_0^2(1 - ux_3x_2) + x_2^2(1+2u(x_1-x_0+x_3)x_3)\\
f_3 &= x_2^2 + u(x_2-x_1) x_0^2x_2 - 2u(x_2-x_1)(x_1 - x_0+x_3)x_2^2,
\end{align*}
where the elements $(x_2 - x_1), (x_1-x_0+x_3), x_3 \in R_{\fp_4}^\times$. Since $f_1 \in I$, we have $
x_2^2 \equiv 2x_0(x_1-x_2) \mod I$, and thus
$$
f_2 \equiv x_0( x_0(1 - ux_3 x_2) + 2(x_1-x_2) + 2u(x_1-x_2)(x_1-x_0+x_3)x_3) \mod I,
$$
but since the element inside the parenthesis is invertible in $R_{\fp_4}$, we conclude that $x_0 \in I$. From $f_1 \in I$, we obtain that $x_2^2 \in I$. Taking this into account, we conclude that
\begin{align*}
f_1 &\equiv -x_2^2 \mod I\\
f_2 &\equiv x_2^2 \mod I\\
f_3 &\equiv x_2^2(1-2ux_2(x_1+x_3)+2ux_1(x_1+x_3)) \equiv x_2^2 \mod I,
\end{align*}
and therefore
$$
(\mathcal{Q}_\sigma)_{\fp_4} \simeq \frac{R_{\fp_4}}{(x_0, x_2^2)} \simeq \kappa\langle 1, x_2\rangle,
$$
with $\length(\mathcal{Q}_\sigma)_{\fp_4} = 2$. Using the associativity formula to compute $m(\sigma)$, we obtain:
$$
m(\sigma) = 1 \cdot 1 + 1 \cdot 1 + 3 \cdot 1 + 2 \cdot 1 = 7.
$$
On the other hand, $e = 2$, so using the Bourbaki formula we obtain $\Bour(\sigma) = 1$, and conclude also the minimal free resolution for $\calt_\sigma$ of the form
$$
0 \to \mathcal{O}_{\mathbb{P}^3}(-4) \to \mathcal{O}_{\mathbb{P}^3}(-2) \oplus \mathcal{O}_{\mathbb{P}^3}(-3)^{\oplus 2} \to \calt_\sigma \to 0.
$$
\end{example}

\begin{prop}\label[prop]{prop:class-pencilcubics-m7}
Let $\sigma = (f,g)$ be a normal, $\mu$-semistable pencil of cubics such that $m(\sigma) = 7$. Then, the only possible cases are:
\begin{itemize}
    \item $e = 2$ and $\sigma$ is a nearly free sequence with $c_3(\calt_\sigma) = 2$ (see \Cref{ex:nearly-free-cubics});
    \item $e = 3$ and $\sigma$ is locally free, $\Bour(\sigma) = 2$ and $B$ is a pair of skew lines (or their degeneration) (see \Cref{ex:pencil-of-cubics-bour2}).
\end{itemize}
\end{prop}

\begin{proof} 
From \Cref{prop:indeg-bounds-pencilcubics}, we know that $e \leq 3$. For $\mu$-semistable sheaves, $e \geq 2$ and, since $c_2(\calt_\sigma(2)) = 1$ and $c_1(\calt_\sigma(2))=0$, from \cite[Lemma 2.1]{chang1984stable}, we conclude that the only possible cases are $c_3 = 2$ (strictly semistable case) or $c_3 = 0$ (stable case). For the strictly semistable case $e = 2$ we obtain $\Bour(\sigma) =1$, so we get the first case.

For the second case, we must have a stable bundle $\calt_\sigma(2)$ with Chern classes $(0, 1, 0)$, which are precisely null correlation bundles described in \cite{wever1977moduli} fitting in a sequence of the form
$$
0 \rightarrow \mathcal{O}_{\mathbb{P}^3}(-1) \rightarrow \calt_\sigma(2) \rightarrow \mathcal{I}_{B}(1) \rightarrow 0
$$
where $B$ is a pair of skew lines or their degeneration, thus $e = \indeg(\calt_\sigma) = 3$.
\end{proof}

Summarizing the main results for normal pencils of cubics, we get the following Theorem:

\phantomsection\label{thm:third}
\begin{maintheorem}\label[maintheorem]{thmC}
\textit{Let $\sigma = (f,g)$ be a normal pencil of cubic surfaces in $\mathbb{P}^3$. Then, if we denote by $e = \indeg(\calt_\sigma)$:
\begin{itemize}
    \item[(a)] $m(\sigma) \leq 12$ and equality holds if and only if $\sigma$ is compressible;
    \item[(b)] The sequence $\sigma$ is free if and only if $m(\sigma) =12, 9$ or $8$, corresponding to $e$ being $0, 1$ or $2$, respectively;
    \item[(c)] There are two nearly-free possible invariants for $\sigma$, either $m(\sigma) = 8, e = 1$ and $c_3 = 4$ (see \Cref{ex:m8-nearlyfree}), or $m(\sigma) = 7$, $e = 2$ and $c_3(\calt_\sigma) = 2$ (see \Cref{ex:nearly-free-cubics});
    \item[(d)] If $m(\sigma) \leq 6$, then $\calt_\sigma$ is $\mu$-semistable, and if $m(\sigma) \leq 2$, then $\calt_\sigma$ is $\mu$-stable.
\end{itemize}}    
\end{maintheorem}

\begin{proof}
Item $(a)$ is \Cref{prop:pencilcubics-compressible} and item $(b)$ with \Cref{prop:class-pencilcubics-m9} and \Cref{ex:free-incompressible-m8}. Item $(c)$ follows from \Cref{prop:class-pencilcubics-m7} and item $(d)$ is \Cref{prop:class-unstable-pencilcubics}.
\end{proof}

Next, we consider some examples. First, a strictly $\mu$-semistable pencil of cubics with $m(\sigma) = 4$.

\begin{example}[$m(\sigma) = 4$, $e = 2$, $\Bour(\sigma) = 4$, $3$-syzygy]\label[example]{ex:strictly-sst-pencil-cubics}
Consider the following pencil of cubics $(d_f = d_g = 2)$, where the first one is smooth:
$$
\sigma = (x_0^3 + x_1^3 + x_2^3 + x_3^3, x_0^3 + x_1^3 + x_2x_3^2),
$$
with Jacobian matrix
$$
\nabla \sigma = 
\begin{pmatrix}
3x_0^2 & 3x_1^2 & 3x_2^2 & 3x_3^2\\
3x_0^2 & 3x_1^2 & x_3^2 & 2x_2x_3
\end{pmatrix}.
$$
The scheme structures $\supp(\mathcal{Q}_\sigma) = \Xi_\sigma$ coincide, with support at a line $V(x_2, x_3)$. Over the prime ideal $\fp = (x_2, x_3)$, we perform some elementary operations for the matrix $\nabla \sigma$ to obtain the form:
$$
(\nabla \sigma)_{\fp} \sim
\begin{pmatrix}
0       & 0 & 3x_2^2 - x_3^2 & 3x_3^2 - 2x_2 x_3\\
ux_0^2  & 1 & (u/3)x_3^2       & (2/3)ux_2 x_3 
\end{pmatrix},
$$
so it sends the second basis vector to $(0,1) \in R_\fp^2$, and then we may compute the cokernel using the first line and remaining columns:
\begin{align*}
(\mathcal{Q}_\sigma)_{\fp} &\simeq \frac{R_\fp}{(3x_2^2 - x_3^2, 3x_3^2 - 2x_2 x_3)}\\
&\simeq \kappa\langle 1, x_2, x_3, x_2^2\rangle,    
\end{align*}
and thus we conclude $\length(\mathcal{Q})_{\fp} = 4$. Since $\deg(R/\fp) = 1$, using the associativity formula we obtain $m(\sigma) = 4$. A first syzygy is $\nu = (-x_1^2, x_0^2, 0, 0)^T$, so we obtain $e = 2$, and thus $\Bour(\sigma) = 4$.

The sheaf $\calt_\sigma$ admits a minimal free resolution of the form:
$$
0 \rightarrow \mathcal{O}_{\mathbb{P}^3}(-6) \to \mathcal{O}_{\mathbb{P}^3}(-4)^{\oplus 2} \oplus \mathcal{O}_{\mathbb{P}^3}(-2) \rightarrow \calt_\sigma \rightarrow 0,
$$
obtained using Macaulay2, so this is an example of a $3$-syzygy sequence, with $c_3(\calt_\sigma) = 16$.
\end{example}

To finish this subsection, we explore in an example that the Bourbaki degree of a pencil of cubics does not depend only on the singularity type of each cubic, but rather how these align when we consider the associated pencil.

\begin{example}\label[example]{ex:same-sing-diff-Bourbaki}
Consider the cubic polynomials:
\begin{align*}
f &= x_3(x_1^2 - x_0 x_2) + x_1^2(x_0 - x_1)\\
g &= x_3(x_1^2 - x_0 x_2) + x_1^3\\
g'&= x_0(x_3^2-x_1x_2) + x_3^3.
\end{align*}
Here, $V(f)$ is a cubic surface with singularity type $4A_1$, with primary ideals $(x_0, x_1, x_2)$, $(x_0, x_1^2, x_3)$ and $(x_1, x_2, x_3)$. On the other hand, the surfaces $V(g)$ and $V(g')$ are both of type $A_12A_2$ and are isomorphic, via the change of coordinates
\begin{align*}
x_0 &\mapsto x_1\\
x_1 &\mapsto x_3\\
x_3 &\mapsto x_0
\end{align*}
and fixing $x_3$. The sequence $\sigma = (f,g)$ is non-normal with $V(x_1)$ in the divisorial component of the Jacobian scheme. On the other hand, $\sigma' = (f,g')$ is a normal, generic pencil with $m(\sigma') = 0$.
\end{example}

\subsection{Degree 6 curves inside quadric surfaces}

Let us focus on the case of normal sequences $\sigma = (f,g)$ with $d_f = 1, d_g = 2$. Here, the Bourbaki degree of a sequence $\sigma$ in terms of $e = \indeg(\calt_\sigma)$ and $ m(\sigma)$ is given by:
$$
\Bour(\sigma) = e(e-3) + 7 - m(\sigma).
$$
We start by applying \Cref{thm:hartshorne1988} to the sheaf $\mathcal{E} = \calt_\sigma$, assuming $\sigma$ is incompressible. The quantities become $A = B = 2$.
Moreover, $c_2(\mathcal{E}) = 7-m(\sigma)$, so direct application of the result shows the following:

\begin{prop}\label[prop]{prop:bounds-c3-mixed-degrees}
Let $\sigma = (f,g)$ be an incompressible normal sequence with $d_f = 1$ and $d_g = 2$. Then $m(\sigma) \leq 5$ and the following hold:
\begin{itemize}
    \item[(a)] If $m(\sigma) = 5$, then $c_3(\calt_\sigma) = 0$.
    \item[(b)] If $0 \leq m(\sigma) < 4$, then $
    c_3 \leq m(\sigma)^2 - 12m(\sigma) + 35$.
\end{itemize}
\end{prop}

From \Cref{thmA}, $(d)$, we obtain:

\begin{prop}\label[prop]{prop:mixeddegrees-compressible}
Let $\sigma$ be a normal sequence with $d_f = 1$, $d_g = 2$. Then $\sigma$ is compressible if and only if $m(\sigma) = 7$ (see \Cref{ex:compressible-mixeddegrees}).
\end{prop}

\begin{example}[Free and compressible sequence with $d_f = 1, d_g = 2$]\label[example]{ex:compressible-mixeddegrees}
Consider the sequence
$$
\sigma = (x_0(x_1 - x_2), x_0^3 + x_1^3 + x_2^3),
$$
which is independent of the variable $x_3$. The matrix
$$
\begin{pmatrix}
0 & -x_0(x_1^2 +x_2^2) \\
0 & x_0^3 + x_1 x_2^2 - x_2^3\\
0 & x_0^3 - x_1^3 + x_1^2 x_2 \\
1 & 0
\end{pmatrix}
$$
gives linearly independent syzygies for $\nabla \sigma$, and thus $\calt_\sigma \simeq \mathcal{O}_{\mathbb{P}^3} \oplus \mathcal{O}_{\mathbb{P}^3}(-3)$.
\end{example}

Next, we apply \Cref{thm:hartshorne1982} to the sheaf $\mathcal{E} = \calt_\sigma(1)$, a sheaf with $c_1 = -1$ and $c_2(\mathcal{E}) = 5 - m(\sigma) \geq 0$. Then, for
\begin{align*}
t &> \sqrt{3c_2 + \frac{1}{4}}-\frac{3}{2} = \sqrt{15 - 3m(\sigma) + \frac{1}{4}}-\frac{3}{2}
\end{align*}
we have $h^0(\mathcal{E}(t)) = h^0(\calt_\sigma(t+1)) \neq 0$. Direct application of this bound shows the following proposition.

\begin{prop}\label[prop]{prop:indeg-bounds-mixed-degrees}
Let $\sigma = (f,g)$ be a normal sequence with $d_f = 1$, $d_g = 2$. Then $e = \indeg(\calt_\sigma) \leq 3$. Whenever $m(\sigma) = 4$, $e \leq 2$ and, whenever $m(\sigma) = 5$, $e \leq 1$.
\end{prop}

We explore some parity results in the next two propositions about local freeness.

\begin{prop}\label[prop]{prop:locfree-notpencil-odd}
Let $\sigma = (f,g)$ be a normal sequence such that $d = d_f + d_g$ is odd. If $\Bour(\sigma)$ is odd, then $\sigma$ is not locally free.
\end{prop}

\begin{proof}
Assuming $\calt_\sigma$ is locally free, we get $c_3(\calt_\sigma) = 0$, and therefore $2p_a(B)-2 = - \Bour(\sigma)(4 + d - 2e)$. But the left-hand-side is even, and the right-hand side is a product of two odd numbers, hence odd, so we get a contradiction.
\end{proof}

\begin{prop}\label[prop]{even-m-mixeddegrees-never-lf}
Let $\sigma$ be an incompressible sequence with $d_f = 1, d_g = 2$ such that $m(\sigma)$ is even, that is, $m(\sigma) \in \{0, 2, 4\}$. Then $\sigma$ is not locally free.
\end{prop}

\begin{proof}
We show that, in any of these cases above, the Bourbaki degree $\Bour(\sigma) = e^2 - 3e + 7-m(\sigma)$ is odd, thus the result follows from \Cref{prop:locfree-notpencil-odd}. Whenever $m(\sigma)$ is even, $7-m(\sigma)$ is odd. We claim $e^2 - 3e$ is always an even number for $e \geq 0$, and therefore $\Bour(\sigma)$ is always odd for $m(\sigma)$ even. First, assuming $e = 2k$ is even, we obtain
$$
e^2 - 3e = 4k^2 - 6k = 2(2k^2 - 3k),
$$
an even number. On the other hand, when $e = 2k+1$ is odd, then
$$
e^2 - 3e = 4k^2 + 4k + 1 - 6k - 3 = 4k^2 - 2k -2 = 2(2k^2 - k - 1),
$$
which is also even.
\end{proof}

\begin{prop}\label[prop]{prop:class-mixeddegrees-m5}
Let $\sigma = (f,g)$ be an incompressible normal sequence with degrees $d_f = 1$, $d_g = 2$. Then $\calt_\sigma \simeq \mathcal{O}_{\mathbb{P}^3}(-1) \oplus \mathcal{O}_{\mathbb{P}^3}(-2)$ if and only if $m(\sigma) = 5$ (see \Cref{ex:mixeddegrees-m5}).
\end{prop}

\begin{proof}
Assuming $\sigma$ is incompressible, by \Cref{prop:indeg-bounds-mixed-degrees}, we obtain that $e = 1$. But from the formula for the Bourbaki degree with $e = 1$, $m(\sigma) = 5$ we obtain $\Bour(\sigma) = 0$, and thus $\sigma$ must be free. On the other hand, if we assume $\calt_\sigma \simeq \mathcal{O}_{\mathbb{P}^3}(-1) \oplus \mathcal{O}_{\mathbb{P}^3}(-2)$, then $e = 1$ and $\Bour(\sigma) = 0$ give $m(\sigma) = 5$.
\end{proof}

\begin{example}[Free and incompressible sequence, $m(\sigma) = 5$]\label[example]{ex:mixeddegrees-m5}
Consider the sequence
$$
\sigma = (x_0x_1, x_3x_2(x_0 - x_1)),
$$
of arrangements of hyperplanes, with Jacobian matrix given by:
$$
\nabla \sigma = \begin{pmatrix}
x_1 & x_0 & 0 & 0\\
x_2x_3 & -x_2x_3 & x_3(x_0-x_1) & x_2(x_0-x_1)
\end{pmatrix}
$$
The matrix
$$
\begin{pmatrix}
0 & x_0(x_0 -x_1) \\
0 & -x_1(x_0-x_1)\\
x_2 & 0 \\
-x_3 & - x_3(x_0 + x_1)
\end{pmatrix}
$$
gives linearly independent syzygies for $\nabla \sigma$, and thus $\calt_\sigma \simeq \mathcal{O}_{\mathbb{P}^3}(-1) \oplus \mathcal{O}_{\mathbb{P}^3}(-2)$, with $e = \indeg(\calt_\sigma) = 1$ and $m(\sigma) = 5$.
\end{example}

\begin{prop}\label[prop]{prop:class-mixeddegrees-m4}
Let $\sigma = (f,g)$ be a normal sequence with degrees $d_f = 1$, $d_g = 2$. If $m(\sigma) = 4$, then $\sigma$ is nearly free, and we have two possible cases:  
\begin{itemize}
    \item[(a)] $\calt_\sigma$ is $\mu$-stable with $e = 2$ and $c_3(\calt_\sigma) = 1$ (see \Cref{ex:nearly-free-mixed-degrees-e2});
    \item[(b)] $\calt_\sigma$ is unstable with $e = 1$ and $c_3(\calt_\sigma) = 3$ (see \Cref{ex:nearly-free-mixed-degrees-e1});
\end{itemize}
Furthermore, these are the only two possibilities of numerical invariants for nearly-free sequences with $d_f = 1, d_g = 2$.
\end{prop}

\begin{proof}
Using \Cref{prop:indeg-bounds-mixed-degrees}, $e \in \{1,2\}$, and both cases imply $\Bour(\sigma) = 1$ when $m(\sigma) = 4$. Using $p_a(B) = 0$ and $\deg(B) = 1$, we obtain the $c_3$'s above, and both can be realized. To conclude the last claim, for $m(\sigma) \leq 3$,
$$
\Bour(\sigma) = 7 - m(\sigma) + e(e-3) \geq 4 + e(e-3) \geq 2
$$
for possible values $1 \leq e \leq 3$.
\end{proof}

\begin{example}[Nearly-free sequence with $d_f = 1, d_g = 2$ and $e = 2$]\label[example]{ex:nearly-free-mixed-degrees-e2}
We consider the following normal sequence with $d_f = 1$, $d_g = 2$:
$$
\sigma = (x_0x_1-x_2x_3, x_1x_3(x_0 - x_2)),
$$
with Jacobian matrix given by
$$
\nabla \sigma = \begin{pmatrix}
x_1 & x_0 & -x_3 & -x_2\\
x_1x_3 & x_3(x_0-x_2) & -x_1x_3 & x_1(x_0-x_2)
\end{pmatrix}.
$$
The irreducible components of $\Xi_\sigma$ consist of three lines $V(x_1, x_3)$, $V(x_1, x_0 - x_2)$ and $V(x_3, x_0 - x_2)$ and a point $p = V(x_2, x_1-x_3, x_0)$ outside the three lines. We compute $m(\sigma) = 4$, where the first line above occurs with multiplicity two and the other two are simple. 

For $\fp_{1} = (x_1, x_3)$, we may consider $u = (x_0)^{-1}$ and perform elementary operations to rewrite the matrix as
$$
(\nabla \sigma)_{\fp_{1}} \sim \begin{pmatrix}
ux_1 & 1 & -ux_3 & -ux_2\\
x_1 x_3(1 - u(x_0-x_2)) & 0 & -x_3(x_1 - x_3(x_0-x_2)u) & (x_0-x_2)(x_1 - x_2x_3u)
\end{pmatrix}.
$$
Thus, it sends the second basis vector $e_2 \in R_{\fp_{1}}^4$ to $(1,0) \in R_{\fp_{1}}^2$, and we may compute the cokernel from the second row of the remaining matrix:
\begin{align*}
(\mathcal{Q}_\sigma)_{\fp_{1}} &\simeq \frac{R_{\fp_{1}}}{(x_1 x_3(1 - u(x_0-x_2)), -x_3(x_1 - x_3(x_0-x_2)u), (x_0-x_2)(x_1 - x_2x_3u))}\\
&\simeq \frac{R_{\fp_{1}}}{(x_1x_3, -x_3^2, x_1 - x_2 x_3 u)}\\
&\simeq \kappa\langle 1, x_3\rangle,
\end{align*}
simplifying the invertible elements $(1 - u(x_0-x_2)), (x_0-x_2) \notin \fp_{1}$, hence $\length((\mathcal{Q}_\sigma)_{\fp_{1}}) = 2$. 

For $\fp_2 = (x_1, x_0-x_2)$, consider $u = (-x_3)^{-1} \in R_{\fp_2}$ and perform elementary operations to rewrite $\nabla \sigma$ as
$$
(\nabla \sigma)_{\fp_2} \sim \begin{pmatrix}
u x_1 & u x_0 & 1 & -ux_2\\
x_1x_3(1 + ux_1) & x_3((x_0-x_2)+ux_0x_1) & 0 & x_1(x_0-x_2)-ux_1x_2x_3
\end{pmatrix}.
$$
Since it sends $e_3$ to $(1,0) \in R_{\fp_2}^2$, we may compute the cokernel using the second row and remaining columns:
$$
(\mathcal{Q}_\sigma)_{\fp_2} \simeq \frac{R_{\fp_2}}{(f_1, f_2, f_3)} \doteq \frac{R_{\fp_2}}{I},
$$
where
\begin{align*}
f_1 &= x_1x_3(1+ux_1)\\
f_2 &= x_3( (x_0-x_2) + ux_0x_1 )\\
f_3 &= x_1( (x_0-x_2) - ux_2x_3).
\end{align*}
Since $f_1 \in I$ and $x_3(1+ux_1) \in R_{\fp_2}^\times$, we obtain $x_1 \in I$. From this and $f_2 \in I$, we obtain:
$$
x_3(x_0-x_2) \in I \Rightarrow (x_0-x_2) \in I,
$$
and thus $I = (x_0-x_2, x_1) = \fp_2 R_{\fp_2}$, so that $(\mathcal{Q}_{\sigma})_{\fp_2} \simeq \kappa$, and thus $\length((\mathcal{Q}_{\sigma})_{\fp_2}) = 1$.

For $\fp_3 = (x_3, x_0-x_2)$, take $u = (x_1)^{-1} \in R_{\fp_3}$ and an elementary operations to obtain 
$$
(\nabla \sigma)_{\fp_3} \sim \begin{pmatrix}
1 & u x_0 & -ux_3 & -ux_2\\
0 & x_3((x_0-x_2)-ux_0x_1) & x_1x_3(-1+ux_3) & x_1(x_0-x_2)+ux_1x_2x_3
\end{pmatrix}.
$$
Since it sends $e_1$ to $(1,0) \in R_{\fp_2}^2$, we may compute the cokernel using the second row and remaining columns:
$$
(\mathcal{Q}_\sigma)_{\fp_2} \simeq \frac{R_{\fp_2}}{(f_1, f_2, f_3)} \doteq \frac{R_{\fp_2}}{I},
$$
where
\begin{align*}
f_1 &= x_3((x_0-x_2)-ux_1x_0)\\
f_2 &= x_1x_3( ux_3 - 1 )\\
f_3 &= x_1( (x_0-x_2)+ux_2x_3 ).
\end{align*}
Since $f_1 \in I$ and $(x_0-x_2) - ux_1x_0 \in R_{\fp_2}^\times$, we obtain $x_3 \in I$. From this and $f_3 \in I$, we obtain:
$$
x_1(x_0-x_2) \in I \Rightarrow (x_0-x_2) \in I,
$$
and thus $I = (x_3, x_0-x_2) = \fp_2 R_{\fp_2}$, so that $(\mathcal{Q}_{\sigma})_{\fp_2} \simeq \kappa$, and thus $\length((\mathcal{Q}_{\sigma})_{\fp_2}) = 1$. From the associativity formula over $\fp_1, \fp_2, \fp_3$, we then obtain $
m(\sigma) = 2 \cdot 1 + 1 \cdot 1 + 1 \cdot 1 = 4$. 

From the Bourbaki degree formula, once we know that $e = 2$, we conclude that $\Bour(\sigma) = 1$ and this is a nearly-free sequence. The minimal free resolution of $\calt_\sigma$ can be obtained computationally
$$
0 \rightarrow \mathcal{O}_{\mathbb{P}^3}(-3) \xrightarrow{\gamma} \mathcal{O}_{\mathbb{P}^3}(-2)^{\oplus 3} \xrightarrow{M} \calt_\sigma \rightarrow 0
$$
given by matrices 
$$
M=\begin{pmatrix}
x_0x_1 +x_3(x_2-x_0) & x_0^2 & x_0x_3\\
-x_1^2+x_1x_3 & -x_0x_1 & 0\\
x_1x_2 & x_2^2 & x_0x_1 + x_3(x_2-x_1)\\
0 & -x_2x_3 & x_1x_3 - x_3^2
\end{pmatrix},
\gamma = \begin{pmatrix}
x_0\\
x_3-x_1\\
-x_2
\end{pmatrix}.
$$
which corresponds to the nearly-free resolution given in \Cref{prop:nearly-free-seq}. We have $c_3(\calt_\sigma) = 1$, corresponding to the point $p$, the unique irreducible component of $\Xi_\sigma$ of codimension three. 
\end{example}

\begin{prop}\label[prop]{prop:class-mixeddegrees-m3}
Let $\sigma = (f,g)$ be a normal sequence with degrees $d_f = 1$, $d_g = 2$, such that $m(\sigma) = 3$ and $\calt_\sigma$ is $\mu$-stable. Then $e = \indeg(\calt_\sigma) = 2$, $\Bour(\sigma) = 2$ and we may have $c_3 = 0, 2, 4$. We have examples for the cases $c_3 = 2$ (see \Cref{ex:B2-mixed-degrees-c3-2}) and $c_3 = 4$ (see \Cref{ex:B2-mixed-degrees-c3-4}).
\end{prop}

\begin{proof}
The sheaf $\calt_\sigma(1)$ is a stable rank two reflexive sheaf with Chern classes $(-1, 2, c_3)$, and $c_3 \leq c_2^2 = 4$, using \cite[Theorem 8.2]{Hartshorne1980}. All three possibilities $c_3 \in \{0, 2, 4\}$ imply $e = 2$. Indeed, this follows from \cite[Proposition 1.1]{sols1981stable} for $c_3 = 0$, \cite[Lemma 2.4]{chang1984stable} for $c_3 = 2$ and \cite[Lemma 9.6]{Hartshorne1980} for $c_3 = 4$.
\end{proof}

Next examples illustrate some cases of the proposition above.

\begin{example}[$m(\sigma) = 3$, $e = 2$, $\Bour(\sigma) = 2$, $c_3 = 2$]\label[example]{ex:B2-mixed-degrees-c3-2}
We consider the following sequence:
$$
\sigma = (x_3(x_0 - x_1), x_0^2x_2+x_0x_1x_3+x_3^3)
$$
with corresponding Jacobian matrix given by:
$$
\nabla \sigma = \begin{pmatrix}
x_3 & -x_3 & 0 & x_0 - x_1\\
2x_0x_2+x_1x_3 & x_0x_3 & x_0^2 & x_0x_1 + 3x_3^2
\end{pmatrix}.
$$
Here, $(\Xi_\sigma)_1$ has two irreducible components, the lines $V(x_3, x_0)$ and $V(x_3, x_0-x_1)$. Let the corresponding primes be denoted by $\fp_1, \fp_2$.

Over the prime $\fp_1 = (x_0, x_3)$, let us take $u \doteq (x_0-x_1)^{-1} \in R_{\fp_1}$ and $Q \doteq x_0x_1 + 3x_3^2$. Consider elementary operations to rewrite
$$
(\nabla \sigma)_{\fp_1} \sim \begin{pmatrix}
ux_3 & -ux_3 & 0 & 1\\
2x_0x_2 + x_1 x_3 - ux_3 & x_0x_3 + uQx_3 & x_0^2 & 0
\end{pmatrix},
$$
so the matrix sends $e_4 \in R_{\fp_1}^4$ to $(1,0) \in R_{\fp_1}^2$, so we may compute the cokernel from the second row and remaining columns:
$$
(\mathcal{Q}_\sigma)_{\fp_1} \simeq \frac{R_{\fp_1}}{f_1, f_2, f_3} \doteq \frac{R_{\fp_1}}{I},
$$
where
\begin{align*}
f_1 &= 2x_0x_2 + x_1 x_3 - u(x_0x_1 + 3x_3^2)x_3\\
f_2 &= x_3(1 + u(x_0x_1 + 3x_3^2))\\
f_3 &= x_0^2.
\end{align*}
Therefore, since $f_2 \in I$ and $u(x_0x_1 + 3x_3^2) \in \fp_1$, we conclude $x_3 \in I$. Rewriting $uf_1$ and removing terms with $x_3$, we obtain $
uf_1 = 2x_0^3x_1x_2$, and since $f_3 = x_0^2 \in I$, we obtain $I = (x_0^2, x_3)$, hence
$$
(\mathcal{Q}_\sigma)_{\fp_1} \simeq \frac{R_{\fp_1}}{(x_0^2, x_3)} \simeq \kappa \langle 1, x_0 \rangle 
$$
so that $\length(\mathcal{Q}_\sigma)_{\fp_1} = 2$. 

Over $\fp_2 = (x_3, x_0-x_1)$, take $u \doteq (x_0^2)^{-1}$ and consider an elementary operation to obtain a matrix
$$
(\nabla\sigma)_{\fp_2} \sim \begin{pmatrix}
x_3 & -x_3 & 0 & x_0-x_1\\
u(2x_0x_2 + x_1 x_3) & u x_0x_3 & 1 & x_0x_1 + 3x_3^2
\end{pmatrix},
$$
which sends $e_3 \in R_{\fp_2}^4$ to $(0,1) \in R_{\fp_2}^2$, so we may compute the cokernel using the remaining columns of the first row, namely:
$$
(\mathcal{Q}_\sigma)_{\fp_2} \simeq \frac{R_{\fp_2}}{(x_3, x_0-x_1)} = \frac{R_{\fp_2}}{\fp_2 R_{\fp_2}} \simeq \kappa,
$$
and therefore $\length(\mathcal{Q}_\sigma)_{\fp_2} = 1$. Thus, from the additivity formula, we obtain $m(\sigma) = 3$. Moreover, $e = 2$, so that $\Bour(\sigma) = 2$. The sheaf $\calt_\sigma$ admits a free resolution of the form:
$$
0 \rightarrow \mathcal{O}_{\mathbb{P}^3}(-5) \rightarrow \mathcal{O}_{\mathbb{P}^3}(-4)^4 \rightarrow \mathcal{O}_{\mathbb{P}^3}(-2) \oplus \mathcal{O}_{\mathbb{P}^3}(-3)^4 \rightarrow \calt_\sigma \rightarrow 0.
$$
so we get $c_3(\calt_\sigma) = 2$ and the Bourbaki scheme is a union of two skew lines.
\end{example}

\begin{example}[$m(\sigma) = 3$, $e = 2$, $\Bour(\sigma) = 2$, $c_3 = 4$]\label[example]{ex:B2-mixed-degrees-c3-4}
We consider the following sequence:
$$
\sigma = (x_0^2 + x_1^2 + x_2^2 + x_3^2, x_3(x_2- x_3)(x_0 - x_1))
$$
with corresponding Jacobian matrix given by
$$
\nabla \sigma = \begin{pmatrix}
2x_0 & 2x_1 & 2x_2 & 2x_3\\
x_3(x_2 - x_3) & x_3(x_3-x_2) & x_3(x_0-x_1) & (x_0-x_1)(x_2 - 2x_3)
\end{pmatrix}.
$$
Here, the irreducible components of $(\Xi_\sigma)_1$ are three lines, given by the associated primes:
\begin{align*}
\fp_1 &= (x_2,x_3)\\
\fp_2 &= (x_0-x_1, x_3)\\
\fp_3 &= (x_2-x_3, x_0-x_1).
\end{align*}
Over $\fp_1$, we may choose $u = (2x_1)^{-1} \in R_{\fp_1}$ to perform the elementary operations and rewrite:
$$
(\nabla \sigma)_{\fp_1} \sim \begin{pmatrix}
2ux_0 & 1 & 2ux_2 & 2ux_3\\
x_3(x_2-x_3)(1 + u x_0) & 0 & x_3( (x_0-x_1)-2ux_2(x_3-x_2) ) & (x_0-x_1)(x_2 - 2x_3) - 2ux_3^2(x_3-x_2)
\end{pmatrix}
$$
Thus, it sends $e_2$ to $(1,0)$ and compute the cokernel by considering the remaining columns of the second row, namely
$$
(\mathcal{Q}_\sigma)_{\fp_1} \simeq \frac{R_{\fp_1}}{I}
$$
with $I = (f_1, f_2, f_3)$ and
\begin{align*}
f_1 &= x_3(x_2-x_3)\\
f_2 &= x_3( (x_0-x_1)-2ux_3(x_3-x_2))\\
f_3 &= (x_0-x_1)(x_2-2x_3)-2ux_3^2(x_3-x_2).
\end{align*}
Since $f_1 \in I$, $x_3(x_2-x_3) \in I$, and we may rewrite $f_2$ as
$$
f_2 = x_3 (x_0-x_1) + 2ux_2x_3(x_2-x_3) \in I \Rightarrow x_3(x_0-x_1) \in I,
$$
and therefore $x_3 \in I$. Now, turning to $f_3 \in I$ and removing terms with $x_3$, we obtain:
$$
f_3 \equiv x_2(x_0-x_1) \Rightarrow x_2 \in I,
$$
since $x_0-x_1$ is invertible. Thus, $I \simeq \fp_1 R_{\fp_1}$ and $\length(\mathcal{Q}_\sigma)_{\fp_1} = 1$.

Over $\fp_2=((x_0-x_1),x_3)$, we may choose the same elementary operations and thus
$$
(\mathcal{Q}_\sigma)_{\fp_2} \simeq \frac{R_{\fp_2}}{I}
$$
with $I = (f_1, f_2, f_3)$, the polynomials as before. Since $f_1 \in I$, we conclude:
$$
f_3 = (x_0-x_1)(x_2-2x_3) +2ux_3^2(x_2-x_3) \Rightarrow (x_0-x_1)(x_2-2x_3) \in I.
$$
Since $(x_2 - 2x_3)$ is invertible in $R_{\fp_2}$, we obtain $(x_0-x_1) \in I$. From this fact, since $x_2 - x_3$ is invertible and $f_1 \in R_{\fp_2}$, we obtain $x_3 \in I$, and thus $I \simeq \fp_2 R_{\fp_2}$ and $\length(\mathcal{Q}_\sigma)_{\fp_2} = 1$.

Over the prime $\fp_3$, we may also do the same elementary operations and consider the same generators for the ideal $I$. From $f_1 \in I$ we get $(x_2-x_3) \in I$, since $x_3$ is now invertible. Moreover, from $f_2 \in I$ we obtain:
$$
f_2 = x_3(x_0-x_1) + 2ux_2x_3(x_2-x_3) \Rightarrow x_3(x_0-x_1) \in I \Rightarrow (x_0-x_1) \in I,
$$
so that $I \simeq \fp_3 R_{\fp_3}$ and $\length(\mathcal{Q}_\sigma)_{\fp_3} = 1$. Thus, from the associativity formula $m(\sigma) = 3$. Moreover, $e = 2$ and thus $\Bour(\sigma) = 2$. Here, the sheaf $\calt_\sigma$ admits a free resolution below
$$
0 \rightarrow \mathcal{O}_{\mathbb{P}^3}(-4) \rightarrow \mathcal{O}_{\mathbb{P}^3}(-2)^2 \oplus \mathcal{O}_{\mathbb{P}^3}(-3) \rightarrow \calt_\sigma \rightarrow 0
$$
so that $c_3(\calt_\sigma) = 4$, and the associated Bourbaki scheme is a plane conic.
\end{example}

\begin{cor}\label[cor]{cor:class-unstable-mixeddegrees}
Let $\sigma$ be a normal sequence with $d_f = 1, d_g = 2$. If $m(\sigma) < 3$, then $\calt_\sigma$ is $\mu$-stable.
\end{cor}

\begin{proof}
Since $\mu(\calt_\sigma) = -3/2$, then $\calt_\sigma$ is $\mu$-stable if and only if $e > 1$. For $m(\sigma) < 3$, $\Bour(\sigma) > 2$ and by \Cref{thmB}, $(a)$, $e \neq 1$.
\end{proof}

We summarize the main results of this subsection in the following theorem:
\phantomsection\label{thm:fourth}
\begin{maintheorem}\label[theorem]{thmD}
\textit{Let $\sigma = (f,g)$ be a normal sequence with $d_f = 1, d_g = 2 $. Then, if we denote by $e = \indeg(\calt_\sigma)$:
\begin{itemize}
    \item[(a)] $m(\sigma) \leq 7$ and equality holds if and only if $\sigma$ is compressible;
    \item[(b)] The sequence $\sigma$ is free if and only if $m(\sigma) = 7$ or $5$, and each corresponds to $e$ being $0$ or $1$, respectively;
    \item[(c)] There are two cases of nearly-free sequences $\sigma$, both with $m(\sigma) = 4$, one where $\calt_\sigma$ is $\mu$-stable with $c_3(\calt_\sigma) = 1$ and another one where $\calt_\sigma$ is $\mu$-unstable with $c_3(\calt_\sigma) = 3$ (see \Cref{ex:nearly-free-mixed-degrees-e2} and \Cref{ex:nearly-free-mixed-degrees-e1});
    \item[(d)] If $m(\sigma) = 3$ then $\Bour(\sigma) = 2$, with a possible unstable case $(c_3 = 6)$ and stable cases with $c_3 \in \{0, 2, 4\}$. Among these, we have examples for the stable cases with $c_3 = 2$ (see \Cref{ex:B2-mixed-degrees-c3-2}) and $c_3 = 4$ (see \Cref{ex:B2-mixed-degrees-c3-4}).
    \item[(e)] If $m(\sigma) < 3$, then $\calt_\sigma$ is $\mu$-stable.
\end{itemize}}
\end{maintheorem} 

\begin{proof}
Item $(a)$ is \Cref{prop:mixeddegrees-compressible} and item $(b)$ follows with \Cref{prop:class-mixeddegrees-m5}. Item $(c)$ is shown in \Cref{prop:class-mixeddegrees-m4}, and item $(d)$ is described in \Cref{prop:class-mixeddegrees-m3}. The stability result in $(e)$ is in \Cref{cor:class-unstable-mixeddegrees}.
\end{proof}

\printbibliography

\end{document}